%% file: Attainment_v4_arxiv_Update.tex
\documentclass{amsart}
\usepackage[utf8]{inputenc}
\usepackage{amsmath}
\usepackage{amssymb}
\usepackage{amsthm}
\usepackage{amsopn}
\usepackage{color}
\usepackage{nicefrac}
\usepackage{enumerate}
\usepackage[hidelinks]{hyperref}

\numberwithin{equation}{section}

\newtheorem{theorem}{Theorem}[section]

\newtheorem{lemma}[theorem]{Lemma}
\newtheorem{definition}[theorem]{Definition}
\newtheorem{proposition}[theorem]{Proposition}

\theoremstyle{remark}
\newtheorem{remark}[theorem]{Remark}

\title[Analysis of attainment of boundary conditions]{Analysis of the attainment of boundary conditions for a nonlocal diffusive Hamilton-Jacobi equation}
\author{Alexander Quaas, Andrei Rodr\'iguez}

\begin{document}

\begin{abstract}
	We study whether the solutions of a parabolic equation with diffusion given by the fractional Laplacian and a dominating gradient term satisfy Dirichlet boundary data in the classical sense or in the generalized sense of viscosity solutions. The Dirichlet problem is well posed globally in time when boundary data is assumed to be satisfied in the latter sense. Thus, our main results are \emph{a)} the existence of solutions which satisfy the boundary data in the classical sense for a small time, for all H\"older-continuous initial data, with H\"older exponent above a critical a value, and \emph{b)} the nonexistence of solutions satisfying the boundary data in the classical sense for all time. In this case, the phenomenon of loss of boundary conditions occurs in finite time, depending on a largeness condition on the initial data.
\end{abstract}

\maketitle



\input{intro_Attainment_v4.tex}

\input{notionSoln_Attainment_v4.tex}

\input{localExt_Attainment_v4.tex}

\input{regularization_Attainment_v4.tex}

\input{technical_Attainment_v4.tex}

\input{nonexistence_Attainment_v4.tex}

\input{appendix_Attainment_v4.tex}
\noindent\textbf{Acknowledgments:} A.Q.~was partially supported by Fondecyt Grant No. 1151180, Programa Basal, CMM, U. de Chile and Millennium Nucleus Center for Analysis of PDE NC130017. A.R.~was partially supported by Conicyt, Beca Doctorado Nacional 2016.

\bibliographystyle{plain}
\bibliography{../../aerp.bib}

\vspace{5mm}
\noindent \textsc{Alexander Quaas}\\
\noindent \textit{Email:} alexander.quaas@usm.cl\\
\noindent \textsc{Andrei Rodríguez}\\
\noindent \textit{Email:} andrei.rodriguez.14@sansano.usm.cl\\[4pt]
\noindent \textsc{Departamento de Matemática, Universidad Técnica Federico Santa María, Casilla: V-110, Avda. España 1680, Valparaíso, Chile.}

\end{document}

%% file: intro_Attainment_v4.tex
\section{Introduction}\label{sec:intro}

The present work is a contribution to the study of the qualitative properties of a nonlinear parabolic equation involving nonlocal diffusion. Specifically, we study the occurrence of \emph{loss of boundary conditions} (LOBC, for short; see Sec.~\ref{notionSolnSec} for a precise definition) for the following problem:
	\begin{align}
		u_t  + (-\Delta)^s u = |Du|^p &\quad\textrm{in } \Omega\times(0,T), \label{mainEq}\\
		u = 0 &\quad\textrm{in } \mathbb{R}^N\backslash\Omega\times(0,T),\label{boundarydata}\\
		u(x,0) = u_0(x) &\quad\textrm{in } \overline{\Omega} \label{initialdata}.
	\end{align}
Here, $\Omega\subset\mathbb{R}^N$ is a bounded domain with $C^2$ boundary, $T>0$, and $(-\Delta)^s$ denotes the well-known fractional Laplacian operator, defined as
	\begin{equation}\label{fracDef}
		(-\Delta)^s u(x,t) = C_{N,s}\, P.V.\!\int_{\mathbb{R}^N} \frac{u(x,t) - u(y,t)}{|x-y|^{N+2s}} \,dy,
	\end{equation}
where $C_{N,s}$ is a normalization constant. See \cite{di2012hitchhikers} for details. In addition to $s\in(0,1)$, we impose the following restrictions on $s$ and $p$,
	\begin{equation}\label{pands}
		s+1<p<\frac{s}{1-s}.
	\end{equation}
In particular, these imply that $s\in (0.618\ldots, 1)$, where $0.618\ldots$ is the constant sometimes called \emph{reciprocal golden ratio}. We provide further comment on these restrictions below, after the statement of our main theorems. Moreover, we assume $u_0\geq 0$,
	\begin{equation}\label{u0reg}
		u_0\in C^\beta(\overline{\Omega}), \quad\textrm{where } \beta>\beta^* = \frac{p-2s}{p-1},
	\end{equation}
and the compatibility condition
	\begin{equation}\label{bottomComp}
		u_0(x) = 0 \quad\textrm{for all } \partial{\Omega}.
	\end{equation}
As with \eqref{pands}, \eqref{u0reg} might not be optimal, and is explained in the context of our Theorem \ref{localExtThm} (see below).

Equation \eqref{mainEq} can be seen as a generalization of the so-called viscous Hamilton-Jacobi equation,
	\begin{equation}\label{vhj}
		u_t - \Delta u = |Du|^p 	\textrm{ in } \Omega\times(0,T).
	\end{equation}
For $p=2$, this corresponds to the deterministic Kardar-Parisi-Zhang equation, proposed by these authors in \cite{kardar1986dynamic} as a model for the profile of a growing interface. Due to its mathematical relevance as a simple model for an equation with nonlinear dependence on the gradient of the solution, \eqref{vhj} has been studied from numerous points of view and with different qualitative results in mind: existence and uniqueness of classical solutions (\cite{friedman2013partial}); existence and nonexistence of global, classical solutions and gradient blow-up and related phenomena (\cite{souplet2002gradient}, \cite{souplet2006global}, \cite{yuxiang2010single}, \cite{porretta2016profile}; see also \cite{quittner2007superlinear} and the references therein for a broader context); global existence of viscosity solutions, assuming boundary conditions in the viscosity sense (\cite{barles2004generalized}); and, closest to our work, regarding LOBC, \cite{porretta2017analysis}.

Some of the previous results have been extended to more general equations, still in the second-order setting: e.g., to degenerate equations in \cite{attouchi2012well}, \cite{attouchi2017single}; and, by the authors, to fully nonlinear, uniformly parabolic equations in \cite{quaas2018loss}, which the present work closely parallels. 

Under the structural assumptions of nondegeneracy of the diffusion and coercivity of the first order term, which are easily shown to be satisfied by \eqref{mainEq} (see Remark \ref{signChange}), and the compatibility condition \eqref{bottomComp}, together with the notion of boundary conditions \emph{in the viscosity sense}, the existence of a unique solution of \eqref{mainEq}-\eqref{boundarydata}-\eqref{initialdata} defined globally in time, $u\in C(\overline{\Omega}\times[0,\infty))\cap L^\infty(\Omega\times(0,T))$ for all $T>0$, is shown in \cite{barles2016existence}, Theorem 7.1. This result is proved by means of a comparison result (\cite{barles2016existence}, Theorem 3.2), and a subsequent application of Perron's method. See Sec.~\ref{notionSolnSec} for a precise definition of the notion of solution employed and further remarks on the application of the results of \cite{barles2016existence} to our problem. 

We remark that there exist certain results concerning the regularity of solutions for problems related to \eqref{mainEq}-\eqref{boundarydata}-\eqref{initialdata} (see, e.g., \cite{barles2015regularity}, \cite{barles2017lipschitz}). However, even if they were adapted to our setting, they do not provide the regularity needed for the proof of Theorem \ref{nonExtThm}. For this reason we resort to a regularization procedure. See the remarks after the statement of Theorem \ref{nonExtThm}.

\subsection{Main results}

The first of our results concerns local existence, i.e., the existence of solutions which, for small time, satisfiy \eqref{boundarydata} in the classical sense.

	\begin{theorem}\label{localExtThm}
		Assume \eqref{pands} and \eqref{u0reg}. Then, there exists a $T^*>0$, depending only on $N, \Omega, s, p$ and $[u_0]_{C^\beta(\overline{\Omega})}$, such that the viscosity solution of \eqref{mainEq} satisfies \eqref{boundarydata} and \eqref{initialdata} in the classical sense for all $0\leq t \leq T^*$.
	\end{theorem}


Due to the results of \cite{barles2016existence}, it suffices to show that the globally defined viscosity solution of \eqref{mainEq} satisfies \eqref{boundarydata} in the classical sense (see Sec.~\ref{notionSolnSec}). This is accomplished by a barrier argument, i.e., the construction of a supersolution of \eqref{mainEq}-\eqref{boundarydata}-\eqref{initialdata} in a neighborhood of $\partial\Omega$. It is here that the restriction \eqref{pands} comes into play. Consider $s\in (0,1)$ fixed. The upper bound $p<\frac{s}{1-s}$ implies for the critical exponent in \eqref{u0reg} that $\beta^*<s$, while the construction of the barrier ultimately relies on computing
	\begin{equation*}
		F(\beta) = \int_{\mathbb{R}} \frac{(1+\beta)_+ + (1-\beta)_+ - 2}{|t|^{1+2\beta}} \,dt
	\end{equation*}
for $\beta\in(0,2s)$; more precisely, on the fact that $F(\beta)<0$ for each $\beta\in (0,s)$ (see \cite{davila2016continuous}, \cite{ros2014dirichlet} for details).

We note that neither the corresponding local existence result for \eqref{vhj} in \cite{friedman2013partial} nor its extension to the fully nonlinear case in \cite{quaas2018loss} require an upper bound for $p$ (or, more generally, for the rate of growth of the gradient nonlinearity). In this sense our result might not be optimal. It is \emph{stable}, however, in the sense that the restriction disappears as we approach the second-order case, since $\frac{s}{1-s}\to\infty$ as $s\to1^-$.

The statement of our second result, concerning LOBC, involves the principal eigenfunction of the fractional Laplacian. We denote by $(\lambda_1, \varphi_1)$ the solution pair for
	\begin{equation}\label{fracLapEigen0}
		\left\{\begin{array}{ll}
			(-\Delta)^s\varphi_1 = \lambda_1\varphi_1& \textrm{in } \Omega,\\
			\varphi_1 = 0 & \textrm{in } \mathbb{R}^N\backslash\Omega,
		\end{array}\right.
	\end{equation}
where the solution is normalized so that $\|\varphi_1\|_\infty = 1$. See Sec.~\ref{techSec} for details.
	\begin{theorem}\label{nonExtThm}
		Assume \eqref{pands} and \eqref{u0reg}, and let $T>0$. Then, there exists $M>0$, depending only on $N, \Omega, s, p$, and $T$, such that, if
		\begin{equation}
			\int_\Omega u_0(x)\varphi_1(x) \,dx > M,
		\end{equation}
		then the viscosity solution of \eqref{mainEq}-\eqref{boundarydata}-\eqref{initialdata} has LOBC at some finite time prior to $T$.
	\end{theorem}
The proof of Theorem \ref{nonExtThm} uses a key argument from Theorem 2.1 in \cite{souplet2002gradient}, sometimes referred to as the ``principal eigenfunction method.'' In adapting this argument to the current setting, the main difficulty is the lack of regularity of solutions. More precisely, we would need solutions to satisfy the equation either in the weak sense, or in some pointwise sense; in the latter case, all terms in the equation must be summable. As mentioned earlier, the existing theory for our problem does not provide such regularity. We note that even in the case of \eqref{vhj} (for which LOBC is obtained in \cite{porretta2017analysis}, among other results), where viscosity solutions are shown to be smooth, some approximation procedure is necessary to ``integrate'' the equation.

We remedy this problem by using regularization by $\inf$-$\sup$-convolutions, a method introduced in \cite{lasry1986remark}. Afterwards, we require that various estimates related to \eqref{fracLapEigen0} remain uniform with respect to the regularization parameters. In particular, we obtain the stability of solutions to \eqref{fracLapEigen0} with respect to the varying domain (see Subsec.~\ref{subsec:stability}). For this part we also rely on fundamental estimates for the Dirichlet problem for the fractional Laplacian from \cite{ros2014dirichlet}.
	
The part of \eqref{pands} that is relevant to Theorem \ref{nonExtThm} is $s+1<p$. This assumption appears only in the crucial Lemma \ref{finiteIntLemma}, which states that a certain negative power (given by $p$ and $s$) of the principal eigenfunction on an approximate domain is summable. The restriction \eqref{pands} is related to our method of proof. However, a lower bound for $p$ in terms of $s$ is necessary for LOBC to occurr: it is known that solutions of \eqref{mainEq} satisfying \eqref{boundarydata} in the classical sense for all $T>0$ exist for $s\in [0, 1]$ and $p\leq 2s$ (see \cite{barles2008dirichlet}, Theorem 4, for the nonlocal case and \cite{barles2004generalized} for the local case, i.e.,  $s=1$). A natural question which we leave open is whether there is classical solvability or if LOBC occurs when $s\in [0, 1)$ and $2s<p\leq s+1.$

For simplicity, we have restricted our analysis to the case of homogeneous boundary conditions, as in \eqref{boundarydata}. Local existence for more general boundary conditions can be obtained in the same way as in Theorem \ref{localExtThm}, following the construction of \cite{davila2016continuous}. Theorem \ref{nonExtThm} applies to the case of general boundary conditions with practically no modification (see Remark \ref{nonExtRmk1}). 

The methods of Theorem \ref{nonExtThm} apply to more general nonlinear operators as well, provided they satisfy the nonlocal equivalent of having divergence form (see \cite{silvestre2006holder}, Sec.~3.6). This restriction is due to the essential use of ``integration by parts'' in the so-called principal eigenfunction method. For instance, the result can be extended to an equation with diffusion given by the so-called $p$-fractional Laplacian, defined for $s\in (0,1)$, $p>1$ and $x\in\mathbb{R}^N$ as
	\begin{equation}\label{pLaplace}
		(-\Delta_p)^s u(x) = C(N,s,p)\, P.V.\int_{\mathbb{R}^N} \frac{|u(x)-u(y)|^{p-2}(u(x)-u(y))}{|x-y|^{N+sp}}\,dy.
	\end{equation}
In this case the technical results of Sec.~\ref{techSec} can be reproduced following \cite{iannizzotto2016global} and \cite{del2017hopf}. 

\subsection{Organization of the article}
In Sec.~{\ref{notionSolnSec}} we recall the notion of solution from \cite{barles2016existence}, which is used throughout our work, and provide some remarks on the relevant results contained in that work. Sec.~{\ref{localExtSec} is devoted to the proof of Theorem \ref{localExtThm}. In Sec.~\ref{techSec} we provide the technical results required for the proof of Theorem \ref{nonExtThm}, which is proved in Sec.~\ref{nonExtSec}.

\subsection{Notation}

We write $d:\mathbb{\mathbb{R}^N}$, $d=d(x)$ for the \emph{distance to the boundary} of the set $\Omega$, extended by zero to the whole of $\mathbb{R}^N$, i.e.,
	\begin{equation}\label{defDist}
		d(x) = \left\{
			\begin{array}{ll}
				\mathrm{dist}(x,\partial\Omega) & \textrm{if } x\in\overline{\Omega},\\
				0 & \textrm{if } x\in\mathbb{R}^N\backslash\overline{\Omega}.
			\end{array}
			\right.
	\end{equation}
	
For $\delta>0$, we write $\Omega_\delta = \{x\in\Omega : d(x)<\delta \},$ where $d=d(x)$ is defined as above. Similarly, $\Omega^\delta = \{x\in\Omega : d(x)>\delta \}.$ However, to avoid confusion, we abstain from using both notations in the same section: e.g., in Sec.~\ref{localExtSec} we use the notation $\Omega_\delta$, and from Sec.~\ref{techSec} onwards we use only $\Omega^\delta$. The closure and boundary operation on sets is performed ``after'' specifying a subset in terms of the distance: e.g., $\overline{\Omega}_\delta = \{x\in\Omega : d(x)\leq\delta \}.$

In Sec.~\ref{techSec} and Appendix \ref{CsReg} we write, for $\eta>0$, $d_\eta=d_\eta(x)$ for the distance to the boundary of $\Omega^\eta$ (extended by zero outside this set, as in \eqref{defDist}).

Nonnegative constants whose precise value does not affect the argument are denoted collectively by $C$, and the value of $C$ may change from line to line. When convenient, dependence of $C$ on certain parameters is indicated in parentheses, e.g., as in \eqref{pLaplace}. Dependency on $\Omega, N, s,$ and $p$ is sometimes omitted for simplicity. Constants we wish to keep track of are numbered accordingly ($c_1, c_2,\ldots, C_0, C_1$, etc.)

%% file: notionSoln_Attainment_v4.tex
\section{Notion of solution}\label{notionSolnSec}
	
	We recall the notion of solution for \eqref{mainEq} given in \cite{barles2016existence}: let $\delta>0$, $\phi\in C^2(B_\delta(x)\times (t-\delta, t + \delta))$ and $v\in \mathbb{R^N}\times\mathbb{R}\to\mathbb{R}$ a bounded measurable function. Define
		\begin{equation}\label{eval}
			E_\delta(v,\phi,x,t) = \phi_t(x,t) + I[B_\delta(x)](\phi(\cdot,t))	+ I[\mathbb{R}^N\backslash B_\delta(x)](v(\cdot,t)) - |D\phi(x,t)|^p,
		\end{equation}
	where, for a measurable $A\subset\mathbb{R}^N$ and a bounded, measurable function $\psi$,
		\begin{equation*}
			I[A](\psi(\cdot,t))	= C_{N,s}\, P.V.\!\int_A \frac{\psi(x,t) - \psi(y,t)}{|x-y|^s} \,dy.
		\end{equation*}
	
	We momentarily consider a non-homogeneous boundary condition 
	\begin{equation}\label{genBoundaryData}
		u(x,t)= g(x) \quad\textrm{for all } (x,t)\in \mathbb{R}^N\backslash\overline{\Omega} \times [0,T],
	\end{equation}
	with $g\in C_b(\mathbb{R}^N\backslash\Omega)$, and define the upper (resp., lower) $g$-extension of $u\in USC(\overline{\Omega}\times [0,T])$ (resp., $v\in LSC(\overline{\Omega}\times [0,T])$) as the function defined in $\mathbb{R}^N\times\mathbb{R}$ as
		\begin{equation*}
			\begin{array}{c}
				u^g(x,t)\\
			\end{array} 
			= \left\{\begin{array}{ll}
					u(x,t) & \textrm{if } (x,t)\in\Omega \times [0,T],\\
					g(x,t) & \textrm{if } (x,t)\in\mathbb{R}^N\backslash\overline{\Omega} \times [0,T],\\
					\max \{u(x,t), g(x,t)\} & \textrm{if } (x,t)\in\partial\Omega\times [0,T]
			\end{array}
			\right. 
		\end{equation*}
	(resp.,
		\begin{equation}\label{upLowExt}
			\begin{array}{c}
				v_g(x,t)\\
			\end{array} 
			= \left\{\begin{array}{ll}
				v(x,t) & \textrm{if } (x,t)\in\Omega \times [0,T],\\
				g(x,t) & \textrm{if } (x,t)\in\mathbb{R}^N\backslash\overline{\Omega} \times [0,T],\\
				\min \{v(x,t), g(x,t)\} & \textrm{if }(x,t)\in\partial\Omega\times [0,T]).
			\end{array}
			\right. 
		\end{equation}

	
	\begin{definition}\label{defSoln}
		A function $u\in USC(\overline{\Omega}\times [0,T])$ \emph{(resp., $v\in LSC(\overline{\Omega}\times [0,T])$)} is a subsolution \emph{(resp. supersolution)} of \eqref{mainEq}-\eqref{genBoundaryData}-\eqref{initialdata} if for every $\phi\in C^2(B_\delta(x)\times (t-\delta, t + \delta))$ with $\delta>0$, and every maximum \emph{(resp. minimum)} point $(x_0,t_0)$ of $u^g - \phi$ \emph{(resp. $v_g - \phi$)} over $B_\delta(x)\times (t-\delta, t + \delta)$, we have
			\begin{align*}
				E_\delta(u^g, \phi, x_0,t_0) \leq 0 & \quad \textrm{if } (x_0,t_0) \in \Omega\times(0,T],\\
				\min\{E_\delta(u^g, \phi, x_0,t_0), u(x_0,t_0) - g(x_0,t_0)\} \leq 0 & \quad \textrm{if }(x_0,t_0)\in \partial\Omega\times(0,T],\\
				\min\{E_\delta(u^g, \phi, x_0,t_0), u(x_0,t_0) - u_0(x_0,t_0)\} \leq 0& \quad \textrm{if } x_0\in \overline{\Omega},\ t_0=0,
			\end{align*}
		where $E_\delta$ is defined as in \eqref{eval}. \emph{(resp.}
			\begin{align*}
				E_\delta(v_g, \phi, x_0,t_0) \geq 0 & \quad \textrm{if } (x_0,t_0) \in \Omega\times(0,T],\\
				\max\{E_\delta(v_g, \phi, x_0,t_0), v(x_0,t_0) - g(x_0,t_0)\} \geq 0 & \quad \textrm{if }(x_0,t_0)\in \partial\Omega\times(0,T],\\
				\max\{E_\delta(v_g, \phi, x_0,t_0), v(x_0,t_0) - u_0(x_0,t_0)\} \geq 0 & \quad \textrm{if } x_0\in \overline{\Omega},\ t_0=0.\textrm{\emph{)}}
			\end{align*}
		A viscosity solution of \eqref{mainEq}-\eqref{genBoundaryData}-\eqref{initialdata} is a function whose upper and lower semicontinuous envelopes are respectively sub- and supersolutions, as defined above. LOBC is said to occur whenever the ``classical'' inequality corresponding to \eqref{genBoundaryData} fails at some point, e.g.,
		if for a subsolution $u$ we have
			\begin{equation*}
				u(x_0,t_0) > g(x_0) \quad\textrm{for some } (x_0,t_0)\in \partial\Omega\times(0,T].
			\end{equation*}
		In this case the generalized condition implies that $E_\delta(u^g, \phi, x_0,t_0)\leq 0$ for any $\phi$ as above.
	\end{definition}
		
	\begin{remark}\label{signChange}
		To be precise, the equation to which the results of \cite{barles2016existence} apply is
			\begin{equation}\label{minusEq}
				u_t +(-\Delta)^s u + |Du|^p = 0\quad \textrm{in } \Omega\times(0,T),
			\end{equation}
		which differs from \eqref{mainEq} only in the sign of the nonlinearity (here $\Omega$, $T$, $s$, and $p$ are as before). This does not in any way complicate the analysis, since a simple sign change allows us to go from one equation to the other; i.e., if $u$ is a subsolution (resp. supersolution) of \eqref{minusEq}, then $\tilde{u}=-u$ is a supersolution (resp. subsolution) of \eqref{mainEq}.
	\end{remark}
	
	\begin{remark}\label{noLOBCbottomAndSup}
		Definition \ref{defSoln} also interprets the initial condition \eqref{initialdata} in the viscosity sense, given that $\overline{\Omega}\times\{t=0\}$ (the ``bottom'' of the domain) is part of the parabolic boundary of $\Omega\times(0,T)$. However, as noted in \cite{barles2016existence}, Lemma 4.1, there is no LOBC on this set. That is, if $u$ and $v$ are respectively a bounded, upper-semicontinuous subsolution and a bounded, lower-semicontinuous supersolution of \eqref{mainEq}-\eqref{boundarydata}-\eqref{initialdata}, then
			\begin{equation*}
				u(x,0) \leq u_0(x) \leq v(x,0)\quad \textrm{for all } x\in\overline{\Omega}.
			\end{equation*}
		
		Similarly, as a consequence of \cite{barles2016existence}, Proposition 4.3, and Remark \ref{signChange} above, there is no LOBC for supersolutions of \eqref{mainEq}-\eqref{boundarydata}-\eqref{initialdata} either. More precisely, if $v$ is a bounded, lower-semicontinuous supersolution of \eqref{mainEq}-\eqref{boundarydata}-\eqref{initialdata}, then
			\begin{equation*}
				v(x,t) \geq 0 \quad \textrm{for all } x\in\partial\Omega,\ t\in[0,T].
			\end{equation*}
	\end{remark}
	
	\begin{remark}\label{uBounded}
		An important consequence of the comparison result of \cite{barles2016existence} is that solutions of \eqref{mainEq}-\eqref{boundarydata}-\eqref{initialdata} are uniformly bounded for all $0\leq t \leq T$. Indeed, from the assumptions on the initial data, we have that $\underline{v} \equiv 0$ and 
			\begin{equation*}
				\bar{v}(x) = \left\{\begin{array}{cl}
					\|u_0\|_\infty &\textrm{if }x\in\overline{\Omega},\\
					0 &\textrm{if }x\in\mathbb{R}^N\backslash\overline{\Omega}.
				\end{array}\right.
			\end{equation*}
		are respectively sub- and supersolutions to \eqref{mainEq}-\eqref{boundarydata}-\eqref{initialdata}. Hence, by comparison $0\leq u(x,t) \leq \|u_0\|_\infty$ for all $(x,t)\in\overline{\Omega}\times(0,T)$.
	\end{remark}

%% file: localExt_Attainment_v4.tex
\section{Local existence}\label{localExtSec}	
	From the discussion in the Introduction and from Remark \ref{noLOBCbottomAndSup}, Theorem \ref{localExtThm} follows if we can construct a suitable \textit{supersolution} satisfying \eqref{boundarydata} in the classical sense. To this end we follow the corresponding construction in \cite{davila2016continuous}, which addresses a similar (stationary) problem. For convenience, we state the key estimates obtained therein.
	
	\begin{lemma}\label{fracDistEstimate}
		Let $\Omega\subset\mathbb{R}^N$ be a bounded, $C^2$ domain and $s\in(0,1)$. Then, there exists a $\delta>0$ such that, for each $0<\alpha<s$ there exists $c_1>0$ such that
			\begin{equation*}
				(-\Delta)^s(d(x)^\alpha) \geq c_1d(x)^{\alpha-2s}\quad \textrm{for all }x\in\Omega_\delta.
			\end{equation*}
		The constant $c_1$ depends on $N,\ s$ and $\alpha$, and is such that $c_1\to 0$ as $\alpha\to s^-$.
	\end{lemma}
	\begin{proof}
		This is a special case of Lemma 3.1 in \cite{davila2016continuous}.
	\end{proof}
	
	\begin{lemma}\label{fracConeEstimate}
		Let $\alpha\in(0,2s)$ and let $y\in\mathbb{R}^N$. Then, there exists a constant $c_2>0$ such that
			\begin{equation*}
				(-\Delta)^s(|\cdot - y|^\alpha) \geq -c_2|x-y|^{\alpha-2s}\quad\textrm{for all }x\in\mathbb{R}^N\backslash\{y\}.
			\end{equation*}
		Moreover, there exists a constant $\bar{c}_2>0$ such that $c_2\leq \bar{c}_2$ as $s\to 1^-$.
	\end{lemma}
	\begin{proof}
		This is a special case of Lemma 3.3 in \cite{davila2016continuous}.
	\end{proof}
	
	\begin{proof}[Proof of Theorem \ref{localExtThm}:] Let $u=u(x,t)$ denote the viscosity solution of \eqref{mainEq}-\eqref{boundarydata}-\eqref{initialdata}, which a priori satisfies \eqref{boundarydata} only in the viscosity sense. Our aim is to construct a function $\bar{v} = \bar{v}(x)$ that satisfies the following:
		\begin{align}
			(-\Delta)^s\bar{v}\geq|D\bar{v}|^p & \quad\textrm{in } \Omega_\delta \times (0,T^*),\label{eqbarv}\\
			\bar{v} = 0 & \quad\textrm{in } \mathbb{R}^N\backslash\Omega \times (0,T^*),\label{0exterior}\\
			\bar{v} \geq u & \quad\textrm{in } \Omega\backslash\Omega_\delta\times (0,T^*)\label{vbeatsu},\\
			\bar{v} \geq u_0 & \quad\textrm{in } \overline{\Omega}_\delta,\label{vbeatsu0}
		\end{align}
	where $\delta>0$ and $T^*>0$ are yet to be determined. Here \eqref{0exterior}, \eqref{vbeatsu} and \eqref{vbeatsu0} are meant in the pointwise sense. Note that $\bar{v}$ is time-independent; time only plays a role in \eqref{vbeatsu}, which can fail for a sufficiently large time.
	
	Let $y\in\partial\Omega$, $\beta^*<\alpha<\min\{\beta,s\}$, and let $\lambda, \mu>0$ to be chosen later. We write $M=[u_0]_{\beta,\overline{\Omega}}$ and define
		\begin{equation}
			v_y(x) = \lambda M|x-y|^\alpha + \mu Md(x)^\alpha.
		\end{equation}
	Note that $v_y$ satisfies 	
		\begin{equation*}
			v_y(x)\geq 0 \quad \textrm{for all } x\in\mathbb{R}^N\backslash\Omega 
		\end{equation*}
	and 
	\begin{equation}\label{vyy}
		v_y(y) = 0.
	\end{equation}
	
	For all $x\in \Omega_\delta$, denote by $\pi(x)$ the unique point of $\partial\Omega$ such that $d(x) = |x-\pi(x)|$; $\pi(x)$ is well defined for small enough $\delta>0$ and smooth $\partial\Omega$. Assume $\delta\leq 1$ and $\mu\geq 1$. Using \eqref{bottomComp} and $\alpha<\beta$, we have, for all $x\in\overline{\Omega}_\delta$,
		\begin{align*}
			v_y(x) \geq{}& \mu Md(x)^\alpha = \mu [u_0]_{\beta, \overline{\Omega}}\,|x-\pi(x)|^\alpha \geq [u_0]_{\beta, \overline{\Omega}}\,|x-\pi(x)|^\beta  \\
			\geq{}& |u_0(x) - u_0(\pi(x))| = u_0(x),
		\end{align*}
	where we have used that $u_0\geq 0$ and $u_0(\pi(x))=0$, by \eqref{bottomComp}. This shows that $v_y$ satisfies \eqref{vbeatsu0}.
	
	Now, set $\lambda>0$ large enough so that
		\begin{equation}\label{critical1}
			\lambda|x-y|^\alpha > |x-y|^\beta \quad\textrm{for all } x\in \Omega\backslash\Omega_\delta.
		\end{equation}
	Thus, reasoning as above,
		\begin{align*}
			v_y(x) > \lambda M|x-y|^\alpha \geq [u_0]_\beta |x-y|^\beta \geq u_0(x) \quad \textrm{for all } x\in\Omega\backslash\Omega_\delta.
		\end{align*}
	Note this inequality is \emph{strict} since we are at a positive distance from the boundary. Since $v_y - u_0$ is strictly positive over the compact set $\Omega\backslash\Omega_\delta$, there exists $\epsilon>0$ such that
		\begin{equation*}
			u_0(x) + \epsilon < v_y(x) \quad \textrm{for all } x\in\Omega\backslash\Omega_\delta.
		\end{equation*}
	Recall that the continuous viscosity solution $u$ satisfies \eqref{initialdata} in the classical sense (see Remark \ref{noLOBCbottomAndSup}). This implies that $u(\cdot, t)\rightarrow u_0$ as $t\rightarrow 0^+$ uniformly over $\overline{\Omega}$. Therefore, there exists $T^*>0$ such that
		\begin{equation*}
		u(x,t) < u_0(x) + \epsilon < v_y(x), \quad\textrm{for all } x\in\Omega\backslash\Omega_\delta \text{ and all } t<T^*.
		\end{equation*}
	This gives \eqref{vbeatsu} for $v_y$.
		
	Applying Lemmas \ref{fracDistEstimate} and \ref{fracConeEstimate}, and using that $d(x)\leq |x-y|$, we have
		\begin{align*}
			(-\Delta)^s v_y(x) \geq{}& \mu Mc_1d(x)^{\alpha-2s} - \lambda Mc_2|x-y|^{\alpha-2s} \\
			\geq{}& Md(x)^{\alpha-2s}(\mu c_1 - \lambda c_2) \qquad\textrm{for all } x\in\Omega_\delta,
		\end{align*} 
	for all sufficiently small $\delta>0$. On the other hand, using that $\alpha<1$ and again that $d(x)\leq |x-y|$, we compute
		\begin{equation*}
			|Dv_y(x)|^p \leq M^pd(x)^{p(\alpha-1)}(\mu + \lambda)^p.
		\end{equation*}
	Combining these estimates, we obtain
		\begin{align*}
			&(-\Delta)^s(v_y(x)) - |Dv_y(x)|^p\\
			&\qquad\geq Md(x)^{\alpha-2s}\left(\mu c_1 - \lambda c_2 - M^{p-1}(\mu + \lambda)^pd(x)^{p(\alpha-1) - (\alpha-2s)}\right).
		\end{align*}
	We now take $\mu>0$ large enough, so that $\mu c_1 - \lambda c_2 > \frac{\mu c_1}{2}$, 
	then take $\delta>0$ small enough, so that
	\begin{equation}\label{critical2}
		M^{p-1}(\mu + \lambda)^p\delta^{p(\alpha-1)-(\alpha-2s)} < \frac{\mu c_1}{4}.
	\end{equation}
	Thus,
		\begin{equation*}
			(-\Delta)^sv_y(x) - |Dv_y(x)|^p \geq Md(x)^{\alpha-2s}\left(\frac{\mu c_1}{4}\right) \geq 0,
		\end{equation*}
	which gives that $v_y$ satisfies \eqref{eqbarv}.
	
	By standard arguments, the function
		\begin{equation*}
			v(x) = \inf_{y\in\partial\Omega} v_y(x)
		\end{equation*}
	is a viscosity supersolution of \eqref{eqbarv}. It also satisfies \eqref{vbeatsu} and \eqref{vbeatsu0}, since these are satisfied by $v_y$ for all $y\in\partial\Omega$. Furthermore, $v$ is continuous across $\partial\Omega$ and, by \eqref{vyy}, satisfies
		\begin{equation*}
			v(x)\leq v_x(x) = 0 \quad\textrm{for all } x\in\partial\Omega.
		\end{equation*}
	Therefore, applying the comparison principle of \cite{barles2016existence} over the domain $\Omega_\delta\times(0,T^*)$ we obtain that $u(x,t) \leq v(x)$ for all $x\in\overline{\Omega}_\delta$, $0\leq t\leq T^*$. Hence we conclude that
		\begin{equation*}
			u(x,t) \leq v(x) = 0 \quad \textrm{for all } x\in\partial\Omega,\ 0\leq t\leq T^*.
		\end{equation*}
	\end{proof}
	
	\begin{remark}\label{localExtCrit}
		Local existence can be proven for initial data with ``critical'' regularity, i.e., $u_0\in C^{\beta^*}(\overline{\Omega})$, $\beta^*=\frac{p-2s}{p-1}$, in exactly the same way, provided $[u_0]_{\beta^*,\overline{\Omega}}$ is sufficiently small. More precisely, we define
			\begin{equation*}
				v_y(x) = M|x-y|^{\beta^*} + \mu M d(x)^{\beta^*},
			\end{equation*}
		with $M=[u_0]_{\beta^*,\overline{\Omega}}$. Proceeding as above, we set $\mu>0$ so that ${\mu c_1 - c_2 > \frac{\mu c_1}{2}}$, and require that 
			\begin{equation*}
				[u_0]_{\beta^*,\overline{\Omega}} \leq \left(\frac{\mu c_1}{2(\mu + 1)^p}\right)^{\frac{1}{p-1}}
			\end{equation*}
		is satisfied, instead of \eqref{critical2}. Note that the estimate corresponding to \eqref{critical1} is satisfied automatically.
	\end{remark}
	
	\begin{remark}
		The estimates from Lemmas \ref{fracDistEstimate} and \ref{fracConeEstimate} are stable as $s\to 1^-$. Therefore, Theorem \ref{localExtThm} implies the existence of local solutions for the Dirichlet problem associated to \eqref{vhj} with $u_0\in C^\beta(\overline{\Omega})$, with $\beta>\beta^*=\frac{p-2}{p-1}$; and, following Remark \ref{localExtCrit}, for $u_0\in C^{\beta^*}(\overline{\Omega})$, provided $[u_0]_{\beta^*, \overline{\Omega}}$ is sufficiently small.
	\end{remark}

%% file: regularization_Attainment_v4.tex
\section{Technical results}\label{techSec}
\subsection{Regularization}\label{regSec}

In this section we use regularization by $\inf$-$\sup$-convolution, introduced in \cite{lasry1986remark}, to obtain a supersolution of \eqref{mainEq} which has the regularity needed for the proof of Theorem \ref{nonExtThm}. This function approximates the viscosity solution $u$ of \eqref{mainEq} uniformly over $\overline{\Omega}\times [0,T]$ for any $T>0$ as the regularization parameters tend to zero.
	\begin{definition}\label{convDef}
		Let $v\in C(\overline{\Omega} \times [0,T])$ and $\epsilon, \kappa > 0$. We define
			\begin{align*}
				& v_{\epsilon, \kappa}(x,t) = \inf_{(y,s)\in \Omega \times (0,T)} \left(v(y,s) + \frac{1}{2\epsilon} |x-y|^2 + \frac{1}{2\kappa}|t-s|^2\right),\\
				& v^{\epsilon}(x,t) = \sup_{y\in \Omega} \left(v(y,t) - \frac{1}{2\epsilon} |x-y|^2\right).
			\end{align*}
		We can also define $v^{\epsilon, \kappa}$ and $v_{\epsilon}$ similarly. Note the use of just one superscript when regularization if performed only in the space variable.
	\end{definition}

We collect a series of well-known facts regarding these operations which will be used shortly.
	\begin{proposition}\label{supInfConvProperties}
		Assume $u\in C(\overline{\Omega} \times [0,T])$, and let $\epsilon, \kappa, \delta > 0$.
		\begin{enumerate}[(i)]
			\item Both operations preserve both pointwise upper and lower bounds, i.e., 
					\begin{align*}
						\inf u \leq u_{\epsilon, \kappa} \leq \sup u,\\
						\inf u \leq u^\epsilon \leq \sup u,
					\end{align*}
				where $\inf$ and $\sup$ are taken over $\Omega\times (0,T)$.
			\item\label{optimalpoints} Let $\epsilon^* = 2\sqrt{\epsilon\|u\|_\infty}$, $\kappa^* = 2\sqrt{\kappa\|u\|_\infty}$, $\Omega^{\epsilon^*} = \{ x \in \Omega\ | \ d(x,\partial \Omega)>\epsilon^*\}$. For all $(x,t)\in \Omega^{\epsilon^*} \times (\kappa^*, T-\kappa^*)$, there exist $(y,s)\in\Omega\times(0,T)$ such that
					\begin{equation*}
						u_{\epsilon, \kappa}(x,t) = u(y,s) + \frac{1}{2\epsilon} |x-y|^2 + \frac{1}{2\kappa}|t-s|^2.
					\end{equation*}
				In other words, the $\sup$ and $\inf$ in the definition of the convolutions are achieved, provided we are at sufficient distance from the boundary.
			\item\label{supinflipschitz} Both $u_{\epsilon, \kappa}$ and $u^{\epsilon, \kappa}$ are Lipschitz continuous in $x$ with constant $\frac{K}{\sqrt{\epsilon}}$, where $K = 2\|u\|_\infty$. That is,
				\begin{equation*}
					\sup_{\substack{x,y\in\Omega\\t\in[0,T]}} \frac{|u(x,t) - u(y,t)|}{|x-y|} \leq \frac{K}{\sqrt{\epsilon}}.
				\end{equation*}
			Similarly, they are Lipschitz continuous in $t$ with constant $\frac{K}{\sqrt{\kappa}}$.
			
			\item\label{unifConverApp} $u^{\epsilon, \kappa}, u_{\epsilon, \kappa} \rightarrow u$ uniformly as $\epsilon, \kappa \rightarrow 0$, and similarly for $u^\epsilon$.
			
			\item\label{twicediff} $u^{\epsilon, \kappa}, u_{\epsilon, \kappa}$ are respectively semiconvex and semiconcave. In particular, they are twice differentiable a.e. That is, there are measurable functions $a:\Omega\times [0,T] \to \mathbb{R}$, ${q:\Omega\times [0,T] \to \mathbb{R}^n}$, ${M: \Omega\times [0,T] \to S(n)}$ such that
				\begin{align*}
					u^{\epsilon, \kappa}(y,s) ={}& u^{\epsilon, \kappa}(x,t) + a(x,t)(s-t) + \langle q(x,t),y-x \rangle\\
					&+ \langle M(x,t)(y-x),y-x \rangle + o(|y-x|^2 + |s-t|). 
				\end{align*}
			We will denote $a=(u^{\epsilon, \kappa})_t,\, q=Du^{\epsilon, \kappa},\, M= D^2 u^{\epsilon, \kappa}$ for simplicity. The same goes for $u_{\epsilon, \kappa}$.
			
			\item\label{hessBounds} With the notation above,
				\begin{align*}
					D^2u_{\epsilon, \kappa} \leq \frac{1}{\epsilon}I \quad \textrm{and} \quad D^2u^{\epsilon, \kappa} \geq -\frac{1}{\epsilon}I  \quad a.e. \text{ in } \Omega\times [0,T].
				\end{align*}	
			\item\label{semigroupprop} $(u_{\epsilon, \kappa})_\delta = u_{\epsilon + \delta, \kappa}$.
			
			\item\label{doubleconvineq} $(u_{\epsilon + \delta, \kappa})^\delta \leq u_{\epsilon, \kappa}$.
			
			\item\label{preservSemiconcavity} The operation $u\mapsto (u_\delta)^\delta$ preserves semiconcavity, i.e., if $u$ is $\frac{1}{2\epsilon}$-semiconcave, then $(u_\delta)^\delta$ is $\frac{1}{2\epsilon}$-semiconcave.
		\end{enumerate}	
	\end{proposition}

	\begin{proof}
		The easier proofs follow more or less directly from the definitions (see e.g., \cite{dragoniintroduction}), while \eqref{semigroupprop} and \eqref{doubleconvineq} may be found in \cite{crandall1996equivalence}; \eqref{preservSemiconcavity} is Proposition 4.5 in \cite{crandall1996equivalence}. Property \eqref{twicediff} uses the well-known theorems of Rademacher and Alexandrov on the differentiability of Lipschitz and convex functions, respectively; see \cite{evans2015measure} and the Appendix of \cite{crandall1992user}.
	\end{proof}

For a given $v\in C(\overline{\Omega} \times [0,T])$, we will obtain a function which is Lipschitz continuous with respect to $t$ and $C^{1,1}$ with respect to $x$, following \cite{crandall1996equivalence}. 

First, denote by $\tilde{v}$ the lower ``$0$-extension'' of $v$, as defined in \eqref{upLowExt}. That is, $\tilde{v}=v_g$ with $g\equiv 0$ (this is only to avoid the notation ``$v_0$''). For $v\in C(\overline{\Omega} \times [0,T])$ with $v|_{\partial\Omega}\equiv 0$, this leads to $\tilde{v}\in BUC(\mathbb{R}^N\times[0,T])$. We remark that this is precisely what a solution of \eqref{mainEq}-\eqref{boundarydata}-\eqref{initialdata} with no LOBC satisfies. We then iterate the convolution operators defined above:
	\begin{equation}\label{infsupConv}
		w:=((\tilde{v}_{\epsilon,\kappa})_\delta)^\delta = (\tilde{v}_{\epsilon + \delta,\kappa})^\delta,
	\end{equation}
where $\epsilon, \delta, \kappa >0$. The expression furthest to the right follows from \eqref{semigroupprop} Proposition \ref{supInfConvProperties}, \eqref{semigroupprop}.

As a first step we recall that $\inf$-convolution (resp., $\sup$-convolution) preserves supersolutions (resp. subsolutions) in the viscosity sense, albeit in a proper subset of the original domain.

	\begin{lemma}\label{firstAppIneq}
		Let $u\in C(\overline{\Omega}\times[0,T])$ be the (unique) viscosity solution of \eqref{mainEq}. Then $\tilde{u}_{\epsilon, \kappa}$ (see Definition \ref{convDef} and subsequent remarks), is a viscosity supersolution of \eqref{mainEq} in $\Omega^{\epsilon^*}\times(\kappa^*, T-\kappa^*]$, where
			\begin{equation}\label{epsStar}
				\epsilon^* = 2\sqrt{\epsilon\|u\|_\infty}, \quad \kappa^*=2\sqrt{\kappa\|u\|_\infty},
			\end{equation}
		and $\|u\|_\infty = \sup_{\Omega\times(0,T)} u$.
	\end{lemma}
	
	\begin{proof}
		This is a time-dependent version of Proposition 5.5 in \cite{caffarelli2009regularity}, in the particular case where the equation in its entirety is translation invariant, i.e., when there is no ``$(x,t)$-dependence''. In this case, the regularized function satisfies exactly the same inequality as the original supersolution.
	\end{proof}
		
	\begin{remark}
		We briefly explain the origin of the constants appearing in Lemma \ref{firstAppIneq}. The proof uses the fact that the infimum in Definition \ref{convDef} is achieved at some $(\hat{y}, \hat{s})\in \mathbb{R}^N\times[0,T]$. A simple computation then shows that $|\hat{y} - x|\leq \epsilon^*$, $|\hat{s} - t|\leq \kappa^*$, as defined in Lemma \ref{firstAppIneq}. Therefore, to ensure that $(\hat{y}, \hat{s})\in \Omega\times(0,T)$, so that we can test the equation at this point, we require that $d(x)>\epsilon^*$ and $\kappa^*<t<T-\kappa^*$.
		
		We also remark that, although slightly different from the one given in Definition \ref{defSoln}, the notion of solution given in \cite{caffarelli2009regularity} is essentially the same when concerning the behavior of either sub- and supersolutions at interior points (in particular, for the purposes of Lemma \ref{firstAppIneq}).
	\end{remark}
	
We now state a key proposition concerning the eigenvalues of a regularized function.
	
	\begin{proposition}\label{D2eigenProp}
		Let $v\in BUC(\mathbb{R}^N)$, $\delta>0$ and suppose that $w=(v_\delta)^\delta$ is differentiable everywhere. If for some $\hat{x}$, $w(\hat{x})<v(\hat{x})$ and $w$ is twice differentiable at $\hat{x}$, then $D^2w(\hat{x})$ has $-\frac{1}{\delta}$ as an eigenvalue.
	\end{proposition}
	
	\begin{proof}
		This is Proposition 4.4 in \cite{crandall1996equivalence}, save for the order in which the inf- and sup-convolutions are performed. The proof is entirely analogous. 
	\end{proof} 
	
	\begin{proposition}\label{secAppIneq}
		Let $\eta>0$, $0<t_0<t_1<T$, and $u$ be a bounded, lower-semicontinuous supersolution of \eqref{mainEq} in $\Omega\times(0,T)$. Then there exist $\epsilon, \delta, \kappa > 0$ such that $w = (\tilde{u}_{\epsilon + \delta,\kappa})^\delta$, as defined above, satisfies
			\begin{equation}\label{AppIneq}
				w_t +(-\Delta)^s w - |Dw|^p \geq 0\quad \textrm{\emph{a.e.}~in } \Omega^\eta\times(t_0,t_1).
			\end{equation}
	\end{proposition}
	
	\begin{proof}
		Let $\epsilon$ and $\kappa$ be respectively so small that $\epsilon^*<\eta$ and ${\kappa^*<\min\{t_0, T-t_1\}}$ (see \eqref{epsStar}). This implies that $\Omega^\eta\subset\Omega^{\epsilon^*}$ and $(t_0,t_1)\subset(\kappa^*, T-\kappa^*)$. Applying Lemma \ref{firstAppIneq}, we have that $\tilde{u}_{\epsilon,\kappa}$ is a viscosity supersolution of \eqref{mainEq} in $\Omega^{\eta}\times(t_0,t_1)$. From Proposition \ref{supInfConvProperties} \eqref{doubleconvineq}, we know that $w\leq \tilde{u}_{\epsilon,\kappa}$ in $\mathbb{R}^N\times[0,T]$, and from Proposition \ref{supInfConvProperties} \eqref{twicediff} and \eqref{preservSemiconcavity}, that $w$ has a second order ``parabolic'' Taylor expansion a.e.~in $\Omega^{\eta}\times(t_0,t_1)$. Let $(\hat{x}, \hat{t})\in \Omega^{\eta}\times(t_0,t_1)$ be any point where such an expansion exists.
		
		Assume first that $w(\hat{x}, \hat{t}) = \tilde{u}_{\epsilon,\kappa}(\hat{x},\hat{t})$. In this case, Proposition \ref{supInfConvProperties} \eqref{twicediff} implies that
			\begin{equation*}
				\phi(x,t) = w(\hat{x}, \hat{t}) + w_t(\hat{x}, \hat{t})(t-\hat{t}) + \langle Dw(\hat{x}, \hat{t}), x-\hat{x} \rangle + \frac{1}{2}\langle D^2w(\hat{x}, \hat{t})(x-\hat{x}), x-\hat{x}\rangle 
			\end{equation*}
		satisfies $\phi(x,t) \leq \tilde{u}_{\epsilon,\kappa}$ over $B_\gamma(\hat{x})\times(\hat{t} - \gamma, \hat{t} + \gamma)$ for some $\gamma>0$. As $\phi\in C^2$, it is a valid test function in the sense of Definition \ref{defSoln} (see this Definition for the notation that follows). Since $\tilde{u}_{\epsilon,\kappa}$ is a viscosity supersolution, this implies $E_\gamma(\tilde{u}_{\epsilon,\kappa},\phi,\hat{x},\hat{t}) \geq 0$. As $w$ is $C^{1,1}$ with respect to $x$, $I[B_\gamma(\hat{x})](w(\hat{x},\hat{t}))$ is classically defined. Moreover, the error from the second order expansion results in a summable term over $B_\gamma$, hence
		\begin{equation}\label{expansionErrorCtrl}
			I[B_\gamma(\hat{x})](w(\hat{x},\hat{t})) = I[B_\gamma(\hat{x})](\phi(\hat{x},\hat{t})) + o(1)\textrm{ as }\gamma \to 0.
		\end{equation}
		On the other hand, ${w\leq\tilde{u}_{\epsilon,\kappa}}$ implies that 
		\begin{equation*}
			I[\mathbb{R}^N\backslash B_\gamma(\hat{x})](w(\hat{x},\hat{t})) \geq I[\mathbb{R}^N\backslash B_\gamma(\hat{x})](\tilde{u}_{\epsilon,\kappa}(\hat{x},\hat{t})). 
		\end{equation*}
		Putting this together, we have
			\begin{equation*}
				w_t(\hat{x},\hat{t}) +(-\Delta)^s w(\hat{x},\hat{t}) - |Dw(\hat{x},\hat{t})|^p + o(1)\geq E_\gamma(\tilde{u}_{\epsilon,\kappa},\phi,\hat{x},\hat{t}) \geq 0.
			\end{equation*}
		We thus obtain \eqref{AppIneq} at $(\hat{x},\hat{t})$ by taking $\gamma\to 0$.
		
		Assume now that $w(\hat{x}, \hat{t}) < \tilde{u}_{\epsilon,\kappa}(\hat{x},\hat{t})$. In this case, by Proposition \ref{D2eigenProp}, $D^2w(\hat{x}, \hat{t})$ has an eigenvalue equal to $-\frac{1}{\delta}$. Denoting the eigenvalues of $D^2w(\hat{x}, \hat{t})$ by $e_i,\ i=1,\ldots, N$, without loss of generality we may write $e_1=-\frac{1}{\delta}$. Recall from Proposition \ref{supInfConvProperties}, \eqref{hessBounds} and \eqref{preservSemiconcavity}, that $e_i\leq \frac{1}{\epsilon}$ for $i=2,\ldots, N$. Since $D^2w(\hat{x},\hat{t})$ is symmetric, there exists an orthogonal matrix $A$ such that $D^2w(\hat{x},\hat{t}) = A^T\textrm{diag}(e_1,\ldots,e_N)A$. Therefore, writing $y=Az$, we have
			\begin{align*}
				\langle D^2w(\hat{x}, \hat{t}) z, z\rangle &{}= \langle A^T\textrm{diag}(e_1,\ldots,e_N)A z, z\rangle  = \langle \textrm{diag}(e_1,\ldots,e_N)Az, Az\rangle \\
				&{}= \sum_{i=1}^{N} e_i\, y_i^2 \leq -\frac{1}{\delta} y_1^2 + \sum_{i=2}^{N} \frac{1}{\epsilon} y_i^2.		
			\end{align*}
		We then use an alternative form of the fractional Laplacian (see, e.g., \cite{di2012hitchhikers}) and the second-oder expansion at $\hat{x}$ to compute, for a sufficiently small $\gamma>0$,
			\begin{align*}
				&(-\Delta)^s w(\hat{x}, \hat{t}) = -\frac{C_{N,s}}{2} \int_{\mathbb{R}^N} \frac{w(\hat{x} + z,\hat{t}) + w(\hat{x} - z,\hat{t}) - 2w(\hat{x},\hat{t})}{|z|^{N+2s}} \,dz\\
				&\quad = -\frac{C_{N,s}}{2} \int_{B_\gamma(0)} \frac{\langle D^2w(\hat{x}, \hat{t}) z, z\rangle + o(|z|^2)}{|z|^{N+2s}} \,dz \\
				&\qquad -\frac{C_{N,s}}{2} \int_{\mathbb{R}^N\backslash B_\gamma(0)}\frac{w(\hat{x} + z,\hat{t}) + w(\hat{x} - z,\hat{t}) - 2w(\hat{x},\hat{t})}{|z|^{N+2s}} \,dz\\
				&\quad =: I_1 + I_2.
			\end{align*}
		The uniform convergence of the regularization (Proposition \ref{supInfConvProperties} \eqref{unifConverApp}) we have, taking sufficiently small regularization parameters, that $\|w\|_\infty \leq \|u\|_\infty + 1$, hence
			\begin{align*}
				I_2 &{}\geq -2C_{N,s}\|w\|_\infty\int_{\mathbb{R}^N\backslash B_\gamma(0)}\frac{1}{|z|^{N+2s}}\,dz\\
				&{} \geq -2C_{N,s}(\|u\|_\infty + 1)\int_{\mathbb{R}^N}\frac{1}{|z|^{N+2s}}\,dz =: -C_0.
			\end{align*}
			
		For $I_1$, the error term from the expansion is controlled as in \eqref{expansionErrorCtrl}. Using also that the change of variables $y=Az$ is orthogonal, we compute
			\begin{align*}
				I_1 &{}\geq -\frac{C_{N,s}}{2} \int_{B_\gamma(0)} \frac{\langle D^2w(\hat{x}, \hat{t}) z, z\rangle}{|z|^{N+2s}} \,dz - 1\\
				&{}\geq \frac{C_{N,s}}{2}\left(\frac{1}{\delta}\int_{B_\gamma(0)} \frac{y_1^2} {|y|^{N+2s}} \,dy -\frac{1}{\epsilon}\sum_{i=2}^{N}\int_{B_\gamma(0)} \frac{y_i^2}{|y|^{N+2s}} \,dy\right) - 1
			\end{align*}
		A standard computation shows that, for $i=1,\ldots, N$,
			\begin{equation*}
				0 < \int_{\mathbb{R}^N} \frac{y_i^2}{|y|^{N+2s}} \,dy =: C_1 < \infty.
			\end{equation*}
		Thus, combining the above estimates we have
			\begin{equation*}
				(-\Delta)^s w(\hat{x}, \hat{t}) \geq \frac{C_{N,s}C_1}{2}\left(\frac{1}{\delta}-\frac{N-1}{\epsilon}\right) - C_0 - 1
			\end{equation*}	
		Recalling the bounds for $w_t$ and $Dw$ given in Proposition \ref{supInfConvProperties} \eqref{supinflipschitz}, we finally obtain
			\begin{align*}
				& w_t(\hat{x},\hat{t}) +(-\Delta)^s w(\hat{x},\hat{t}) - |Dw(\hat{x},\hat{t})|^p\\
				&\quad \geq -\frac{K}{\kappa^{\frac{1}{2}}} + \frac{C_{N,s}C_1}{2}\left(\frac{1}{\delta} - \frac{N-1}{\epsilon}\right) - \frac{K^p}{\epsilon^{\frac{p}{2}}} - C_0 - 1.
			\end{align*}
		Therefore, taking $\delta$ sufficiently small with respect to $\kappa$ and $\epsilon$, the right-hand side of the above inequality becomes nonnegative. Hence, we obtain \eqref{AppIneq} at $(\hat{x},\hat{t})$ once more, and the result follows.
	\end{proof}
	
	\begin{remark}
		The first part of the proof of Proposition \ref{secAppIneq}, i.e., working under the assumption $w(\hat{x},\hat{t})=\tilde{u}_{\epsilon,\kappa}(\hat{x},\hat{t})$, amounts to showing the equivalence of using punctually-$C^{1,1}$ test functions instead of $C^2$ in Definition \ref{defSoln}. See Lemma 4.3 in \cite{caffarelli2009regularity} for the time-independent case.
	\end{remark}

%% file: technical_Attainment_v4.tex
\subsection{The principal eigenvalue problem for the Dirichlet fractional Laplacian}\label{subsec:stability}
In this section we provide some results regarding the principal eigenvalue problem for the fractional Laplacian on domains approximating $\Omega$, i.e.,
\begin{equation}\label{fracLapEigenval}
	\left\{\begin{array}{ll}
		(-\Delta)^s\varphi = \lambda \varphi & \textrm{in } \Omega^\eta,\\
		\varphi = 0 & \textrm{in } \mathbb{R}^N\backslash\Omega^\eta.
	\end{array}\right.
\end{equation}

The existence of a solution pair $(\lambda_1^\eta, \varphi_1^\eta)$ of \eqref{fracLapEigenval} where $\lambda_1^\eta>0$, and $\varphi_1^\eta$ is nonnegative in $\Omega^\eta$ and unique up to a multiplicative constant is proved in \cite{servadei2013variational}, Proposition 9. The solution obtained in this work is in the weak, or variational, sense. In particular, $\varphi_1^\eta \in H^s(\Omega^\eta)$. Furthermore, in \cite{servadei2013brezis}, Proposition 4, it is proved that $\varphi_1^\eta \in L^\infty(\Omega^\eta).$ We set
	\begin{equation}\label{normaliz}
		\|\varphi_1^\eta\|_\infty=1 \quad\textrm{for all }\eta>0.
	\end{equation} 
From here, it is possible to apply the results of \cite{ros2014dirichlet} to obtain that $\varphi_1^\eta \in C^s(\mathbb{R}^N)$. Once the ``right-hand side'' of \eqref{fracLapEigenval} is continuous, the notions of weak and viscosity solution coincide (see \cite{ros2014dirichlet}, Remark 2.11). Moreover, ``bootstrapping'' the results contained in \cite{ros2014dirichlet} (see also \cite{caffarelli2009regularity}), the solution can be shown to be regular enough in the interior of $\Omega^\eta$ for \eqref{fracLapEigenval} to hold in a classical, pointwise sense.

Additionally, since $\partial\Omega^\eta$ is smooth (see Remark \ref{choiceEta0}), it can be shown that as a consequence of Hopf's lemma and the strong maximum principle (\cite{greco2016hopf}) that

\begin{equation}\label{lambda1}
	\lambda_1^\eta = \sup \{\lambda > 0 \ | \ \exists\, \varphi > 0 \textrm{ in } \Omega^\eta \textrm{ such that } (-\Delta)^s\varphi \geq \lambda \varphi\}.
\end{equation}

\begin{remark}\label{choiceEta0}
	Towards the proof of our main result, our aim is to provide estimates that remain uniform with respect to the varying domain (i.e, independent of $\eta$). To this end, we recall a few basic facts concerning the geometry of the domains $\Omega^\eta$, $\eta>0$. First, since $\Omega$ is $C^2$ by assumption, there exists an $\eta_0>0$ such that the distance function $d|_{\Omega\backslash\overline{\Omega}\,\!^{2\eta_0}}$ is $C^2$; in particular $\Omega^\eta$ is $C^2$ for all $\eta\in (0,2\eta_0)$.
\end{remark}

The following is a classical result from the geometry of hypersurfaces.

\begin{proposition}\label{parallelCurv}
	Let $\kappa_i$ and $\kappa^\eta_i$, $i=1,\ldots, N-1,$ denote the principal curvatures of $\partial\Omega$ at $y_0\in\partial\Omega$ and of $\partial\Omega^\eta$ at $y = y_0 + \eta\nu(y_0)$, respectively. Then
	\begin{equation}\label{parallelCurveEq}
	\kappa_i^\eta = \frac{\kappa_i}{1 - \eta\kappa_i}.
	\end{equation}
\end{proposition}

\begin{proof}
	See, e.g., \cite{hicks1965notes}, Chap.~2.
\end{proof}

\begin{remark}\label{uniUniExtBall}
	In particular, a $C^2$ domain satisfies an exterior uniform sphere condition. Furthermore, at any given point of the boundary, the radius of the exterior tangent sphere is bounded by below by the smallest of the radii of curvature, which are equal to the inverses of the principal curvatures. Proposition \ref{parallelCurv} allows us to extend the uniform exterior sphere condition to domains close to $\Omega$ \emph{in a uniform way}. More precisely, there exists a positive constant $\rho_0$, depending only on $\Omega$ and $\eta_0$, as given by Remark \ref{choiceEta0}, such that for all $\eta\in (0,\eta_0)$, and for all $y\in\partial\Omega^\eta$, there exists $y_1\in\mathbb{R}^N\backslash\Omega^\eta$ such that $\overline{B_{\rho_0}(y_1)}\cap\overline{\Omega^\eta} = \{y\}.$
\end{remark}

\subsubsection{Stability of eigenfunctions}

\begin{theorem}\label{globalReg}
	Let $\eta_0$ be as in Remark \ref{choiceEta0}. Then, there exists $C$ depending only on $\Omega,\ N$, and $s$, such that, for all $\eta\in(0,\eta_0)$, the positive solution of $\eqref{fracLapEigenval}$, normalized as above, satisfies
		\begin{equation*}
			\|\varphi_1^\eta\|_{C^s(\mathbb{R}^N)}\leq C.
		\end{equation*}
\end{theorem}

\begin{proof}
	Theorem \ref{globalReg} follows readily from the corresponding estimate for the Dirichlet problem. Indeed, we apply Proposition 1.1 from \cite{ros2014dirichlet} to the solution of \eqref{fracLapEigenval}, recalling the normalization \eqref{normaliz}, and obtain
	\begin{equation}\label{phi1Cs}
			\|\varphi_1^\eta\|_{C^s(\mathbb{R}^N)}\leq C_2\|\lambda_1^\eta\varphi_1^\eta\|_{L^\infty(\mathbb{R}^N)} = C_2\lambda_1^\eta \leq C_2\lambda_1^{\eta_0}
	\end{equation}
	for some positive $C_2$ depending on $N$, $s$, and $\Omega^\eta$. For the last inequality we used the fact that $\Omega^{\eta_0}\subset\Omega^\eta$ for $\eta<\eta_0$, and therefore $\lambda_1^\eta \leq \lambda_1^{\eta_0}$, by \eqref{lambda1}.
	
	It remains only to verify that, once $\eta_0$ is fixed, the constant $C_2=C_2(N,s,\Omega^\eta)$ in \eqref{phi1Cs} (i.e., the constant in Proposition 1.1 from \cite{ros2014dirichlet}) can be taken uniformly for $\eta\in(0,\eta_0)$. We perform this analysis in Appendix \ref{CsReg}.
\end{proof}


\begin{theorem}[Stability of eigenvalues]\label{stability}
	Let $\lambda_1$ and $\varphi_1$ denote the principal eigenvalue for $(-\Delta)^s$ and the associated eigenfunction with $\varphi_1>0$ in $\Omega$ and $\|\varphi_1\|_\infty = 1$, respectively. Then, as $\eta\to0$, $\lambda_1^\eta\to\lambda_1$ and $\varphi_1^\eta\to\varphi_1$ uniformly in $\overline{\Omega}$.
\end{theorem}

\begin{proof}
	By \eqref{lambda1}, we have that $\lambda_1^\eta\to\hat{\lambda}$ for some $\hat{\lambda}\geq\lambda_1$. Using Theorem \ref{globalReg} and \eqref{normaliz}, by compactness we have that $\varphi_1^\eta\to\hat{\varphi}$ for some $\hat{\varphi}\in C(\mathbb{R}^N)$, locally uniformly in $\mathbb{R}^N$. In particular, $\varphi_1^\eta\to\hat{\varphi}$ uniformly in $\overline{\Omega}$, and $\hat{\varphi}\equiv0$ in $\mathbb{R}^N\backslash\Omega$. Also, we have $\hat{\varphi}\geq0$ and, from $\eqref{normaliz}$, that $\|\hat{\varphi}\|_\infty=1$. In particular, $\hat{\varphi}\not\equiv 0$.
	
	Let $U\subset\subset\Omega$ and $\eta'>0$ small enough so that $U\subset\Omega^{\eta'}$ (hence, by \eqref{lambda1}, $U\subset\subset\Omega^\eta$ for all $\eta<\eta'$). Then the equation in $\eqref{fracLapEigenval}$ is satisfied in $U$ for all $\eta<\eta'$. Together with the considerations of the preceding paragraph, by the stability of viscosity solutions (see, e.g., Corollary 4.6 in \cite{caffarelli2009regularity}), we have that $(-\Delta)^s\hat{\varphi} = \hat{\lambda}\hat{\varphi}$ in $U$. Furthermore, by the strong maximum principle, 
	we have that $\hat{\varphi}>0$ in $U$. Since the choice of $U$ was arbitrary, the above argument implies that
	\begin{equation*}
		\left\{\begin{array}{ll}
		(-\Delta)^s\hat{\varphi} = \hat{\lambda} \hat{\varphi} & \textrm{in } \Omega,\\
		\hat{\varphi} = 0 & \textrm{in } \mathbb{R}^N\backslash\Omega.
		\end{array}\right.
	\end{equation*}
	By \eqref{lambda1}, this implies $\lambda_1\geq\hat{\lambda}$. Hence, $\hat{\lambda}=\lambda_1$, and, by uniqueness, $\hat{\varphi}=\varphi_1.$
\end{proof}

\begin{remark}\label{uniLowBoundVarphi}
	The following is a consequence of Lemma \ref{stability} that will be useful later on: given $K\subset\subset\Omega$ and $\eta'>0$ small enough, there exists a positive constant $c$, depending on $K$ and $\eta'$, such that, for all $\eta\in(0,\eta')$, $\varphi_1^\eta(x)>c$ for all $x\in K$.
\end{remark}

\subsubsection{A uniform Hopf's lemma}

\begin{lemma}\label{uniBarrier}
	Let $\eta_0$ be as in Remark \ref{choiceEta0}. Then, there exists $C_3$, depending only on $N, s,$ and $\Omega$, such that, for all $\eta\in (0,\eta_0)$,
	\begin{equation}\label{uniBarrierEq}
		(-\Delta)^s d_\eta(x)^s = f_\eta(x) \quad\textrm{for all } x \in \Omega,\ \eta<d(x)<2\eta_0,
	\end{equation}
	for some $f_\eta\in L^\infty(\Omega^\eta\backslash\overline{\Omega}\,\!^{2\eta_0})$ with $\|f_\eta\|_\infty\leq C_3$.
\end{lemma}

For a fixed domain, this is contained in Lemma 3.9 in \cite{ros2014dirichlet}. For completeness, we go into the details of the proof of this result to show that the right-hand side of \eqref{uniBarrierEq} is uniformly bounded. To this end, we also use elements from the proof of the corresponding result for the (more general) case of the fractional $p$-Laplacian in \cite{iannizzotto2016global} (Theorem 3.6). Additionally, we employ Proposition \ref{parallelCurv}.

\begin{proof}[Proof of Lemma \ref{uniBarrier}]
	Fix $\eta>0$ and $y\in\partial\Omega^\eta$. Through a covering argument it suffices to obtain \eqref{uniBarrierEq} in a set $(\Omega^\eta\backslash\overline{\Omega}\,\!^{2\eta_0}) \cap B_{\rho}(y)$ for some $\rho>0$. Since $\partial\Omega^\eta$ is $C^2$ for small enough $\eta>0$, there exists a change of variables $\Psi^\eta\in C^2(\mathbb{R}^N,\mathbb{R}^N), \Psi^\eta(X)=x$, such that $\Psi^\eta=I$ in $\mathbb{R}^N\backslash B_{\rho}(y)$ and
	\begin{equation}\label{XN}
		d_\eta(x) = d_\eta(\Psi^\eta(X)) = X_N, \quad \textrm{for all } x\in (\Omega^\eta\backslash\overline{\Omega}\,\!^{2\eta_0}) \cap B_{\rho}(y).
	\end{equation}
	From this change of variables we define $v_\eta(x):=((\Psi^\eta)^{-1}(x) \cdot e_N)_+^s = d_\eta(x)^s$ and obtain \eqref{uniBarrierEq}, with
	\begin{equation}\label{rhsEta}
		\|f_\eta\|_\infty\leq C(\|D\Psi^\eta\|_\infty, \|(D\Psi^\eta)^{-1}\|_\infty, \rho)\|D^2\Psi^\eta\|_\infty.
	\end{equation}
	
	By further rotation and translation we can assume $y = \Psi^\eta(0)$,
	\begin{equation}\label{Xi}
		\Psi^\eta_i(X) = X_i, \quad\textrm{for } i=1, \ldots, N-1,
	\end{equation}
	and set up a principal coordinate system at $y\in \partial\Omega^\eta$ (see, e.g., \cite{gilbarg2015elliptic}, Sec. 14.6). In these coordinates, we have (using the notation of Lemma \ref{parallelCurv})
	\begin{equation}\label{Dpsix}
		D\Psi^\eta(x) = \mathrm{diag}[1-d_\eta(x)\kappa_1^\eta, \ldots, 1-d_\eta(x)\kappa_{N-1}^\eta, 1]
	\end{equation}
	for $x = y + d_\eta(x)\nu^\eta(y)$. This implies that
	\begin{equation}\label{DPsi}
		\|D\Psi^\eta\|_\infty, \|D(\Psi^\eta)^{-1}\|_\infty \leq C(\kappa_1^\eta,\ldots,\kappa_{N-1}^\eta).
	\end{equation}	
	From \eqref{Xi}, we have $\frac{\partial^2\Psi_k^\eta}{\partial X_i\partial X_j} = 0$ for $k\neq N$; $i,j=1,\ldots, N$. Implicit differentiation of \eqref{XN} yields
	\begin{align}\label{D2Psi}
		&\frac{\partial^2d}{\partial x_j\partial x_N}(\Psi^\eta(X))\frac{\partial\Psi^\eta_N}{\partial X_i} + \frac{\partial^2d}{\partial x_N^2}(\Psi^\eta(X))\frac{\partial\Psi^\eta_N}{\partial X_j}\frac{\partial\Psi^\eta_N}{\partial X_i} + \frac{\partial d}{\partial x_n}(\Psi^\eta(X))\frac{\partial^2\Psi^\eta_N}{\partial X_j\partial X_i} = 0,\nonumber\\[2pt]
		&\hskip 0.562\textwidth\textrm{for } i=1,\ldots, N;\ j=1,\ldots, N-1;\nonumber\\
		&\frac{\partial^2d}{\partial x_N^2}(\Psi^\eta(X))\frac{\partial\Psi^\eta_N}{\partial X_N}\frac{\partial\Psi^\eta_N}{\partial X_i} + \frac{\partial d}{\partial x_n}(\Psi^\eta(X))\frac{\partial^2\Psi^\eta_N}{\partial X_N^2} = 0,\quad \textrm{for } i=1,\ldots, N.
	\end{align}
	We also know, for $x$ as in \eqref{Dpsix}, that
	\begin{equation}\label{D2d}
		D^2d (x) = \mathrm{diag}\left[\frac{\kappa_1^\eta}{1 - d_\eta(x)\kappa_1^\eta}, \ldots, \frac{\kappa_{N-1}^\eta}{1 - d_\eta(x)\kappa_{N-1}^\eta}, 1\right].
	\end{equation} 
	Since $\frac{\partial d}{\partial x_n}(y) = \frac{\partial d}{\partial x_n}(\Psi^\eta(0)) = 1$, by continuity we have that 
	\begin{equation*}
		\frac{\partial d}{\partial x_n}(x) = \frac{\partial d}{\partial x_n}(\Psi^\eta(X)) \neq 0 \quad\textrm{for all } x\in B_\rho(y). 
	\end{equation*}
	This, together with \eqref{DPsi}, \eqref{D2Psi} and \eqref{D2d}, implies that
	\begin{equation*}
		\|D^2\Psi^\eta\|_\infty \leq C(\kappa_1^\eta,\ldots,\kappa_{N-1}^\eta).
	\end{equation*}
	Combining the above estimates with \eqref{rhsEta} and Proposition \ref{parallelCurv}, we conclude that $\|f_\eta\|_\infty$ is uniformly bounded by a constant that depends only on $N,s$, and $\Omega$;  more specifically, on $N, s,$ and the the principal curvatures of $\partial\Omega$.
\end{proof}

The following computation is adapted from \cite{del2017hopf}.

\begin{lemma}[Uniform Hopf's Lemma]\label{uniHopf}
	Let $\eta_0$ be as in Remark \ref{choiceEta0}. Then, there exists a positive constant $c_3$, depending only on $N, s,$ and $\Omega$, such that for all ${\eta\in (0,\eta_0)}$, the solution of \eqref{fracLapEigenval} satisfies
	\begin{equation*}
	\varphi_1^\eta(x) \geq c_3 d_\eta(x)^s \quad\textrm{for all } x\in\Omega,\ \eta\leq d(x) < \eta_0. 
	\end{equation*}
\end{lemma}

\begin{proof}
	Write $K_0=\overline{\Omega}\,\!^{\eta_0}$ for short and define, for $A>0$,
		\begin{equation*}
			v(x) = d_\eta(x)^s + A\chi_{K_0}(x),
		\end{equation*}
	where $\chi_{K_0}$ denotes the characteristic function of $K_0$. Note that $v\in USC(\mathbb{R}^N)$ and $v\in C(\mathbb{R}^N\backslash K_0)$. From Lemma \ref{uniBarrier}, for $x\in\Omega\cap\{\eta<d(x)<\eta_0\}$ we have
		\begin{equation*}
			(-\Delta)^s v(x) = f_\eta(x) + A\,C_{N,s}\,P.V.\int_{\mathbb{R}^N} \frac{-\chi_{K_0}(y)}{|x-y|^{N+2s}}\,dy \leq C_3 - A C(K_0),
		\end{equation*}
	where $C_3$ is the constant from Lemma \ref{uniBarrier} and $C(K_0)>0$. Hence, for sufficiently large $A$, $(-\Delta)^s v(x)\leq 0$.
	
	As a consequence of Theorem \ref{stability}, $\varphi_1^\eta$ is bounded by below on compact sets, uniformly in $\eta$ (see Remark \ref{uniLowBoundVarphi}). We therefore take $c_3>0$ small enough so that
		\begin{equation}
			c_3(\mathrm{diam}(\Omega)^s + A) \leq \inf_{K_0} \varphi_1^\eta.
		\end{equation} 
	Then, for all $\eta\in(0,\eta_0)$,
		\begin{equation}
			(-\Delta)^s (c_3 v) \leq 0 \leq \lambda_1^\eta\varphi_1^\eta \quad\textrm{in } \Omega\cap\{\eta<d(x)<\eta_0\},
		\end{equation}
	and
		\begin{equation} 
		c_3 v = c_3(d_\eta^s + A\chi_{K_0})\leq c_3(\mathrm{diam}(\Omega)^s + A) \leq \inf_{K_0} \varphi_1^\eta \leq \varphi_1^\eta \quad\textrm{in }\overline{\Omega}\,\!^{\eta_0}.
		\end{equation}
	Since $v \equiv \varphi_1^\eta \equiv 0$ in $\mathbb{R}^n\backslash\Omega^\eta$, by comparison we have that
		\begin{equation*}
			c_3v\leq \varphi_1^\eta\quad\textrm{ in }\Omega\cap \{\eta\leq d(x)<\eta_0\}.
		\end{equation*}
	We conclude by observing that $v\equiv d_\eta(\cdot)^s$ in $\mathbb{R}^N\backslash\overline{\Omega}\,\!^{\eta_0}$.
\end{proof}

\begin{lemma}\label{finiteIntLemma}
	Let $\eta_0$ be as in Remark \ref{choiceEta0}. Then, there exists a positive constant $C_4$, depending only on $\Omega, N, s,$ and $p$, such that, for all $\eta\in(0,\eta_0)$,
	\begin{equation*}
		\int_{\Omega^\eta} (\varphi_1^\eta)^{-\frac{1}{p-1}} \,dx \leq C_4.
	\end{equation*}
\end{lemma}
\begin{proof}
	By Lemma \ref{uniBarrier}, we have, for $\eta_0$ and $c_3$ as before,
	\begin{equation*}
		\varphi_1^\eta(x)^{-\frac{1}{p-1}} \leq (c_3 d_\eta(x)^s)^{-\frac{1}{p-1}} \leq Cd_\eta(x)^{-\frac{s}{p-1}} \quad\textrm{for all } x\in\Omega^\eta\backslash\Omega^{\eta_0}.
	\end{equation*}
	Using the change of variables $\Psi^\eta\in C^2(\mathbb{R}^N, \mathbb{R}^N)$ introduced in Lemma \ref{uniBarrier}, together with a covering argument and the estimates provided therein, we have
	\begin{equation*}
		\int_{\Omega^\eta\backslash\Omega^{\eta_0}} d_\eta^{-\frac{s}{p-1}} \,dx = C\int_{\eta}^{\eta_0 - \eta} X_N^{-\frac{s}{p-1}} dX_N \leq C\int_{0}^{\eta_0} X_N^{-\frac{s}{p-1}} \,dX_N.
	\end{equation*}
	By assumption \eqref{pands}, $-\frac{s}{p-1}>-1$, hence this last integral is finite.
	
	On the other hand, using that $\Omega^{\eta_0}\!\!\subset\subset\!\Omega$, 
	together with Lemma \ref{stability} and Remark \ref{uniLowBoundVarphi}, we have $\varphi_1^\eta(x)>c_4>0$ for all $x\in\Omega^{\eta_0}$, where $c_4$ depends only on $\Omega,\ N$ and $s$ (through $\eta_0$). Hence, 
	\begin{equation*}
		\int_{\Omega^\eta} (\varphi_1^\eta)^{-\frac{1}{p-1}} \,dx \leq C\int_{0}^{\eta_0} X_N^{-\frac{s}{p-1}} \,dX_N + \int_{\Omega^{\eta_0}} c_4^{-\frac{1}{p-1}} \,dx =: C_4.
	\end{equation*}
\end{proof}

%% file: nonexistence_Attainment_v4.tex
\section{Nonexistence of global solutions and LOBC}\label{nonExtSec}

	As in our previous work \cite{quaas2018loss}, the proof of Theorem \ref{nonExtThm} uses key ideas from that of Theorem 2.1 in \cite{souplet2002gradient}. For completeness, we reproduce some of the elements of \cite{quaas2018loss}. We remark that some care is required in choosing certain parameters appearing in our argument in the correct order, a difficulty which is not present in \cite{souplet2002gradient}. Specifically, we first choose $u_0$ large in an appropriate sense, then take $\eta$ (which depends on $u_0$ and the regularization parameters of Sec.~\ref{techSec}) sufficiently small.
	
	\begin{proof}[Proof of Theorem \ref{nonExtThm}]
	Consider the differential inequality
		\begin{equation}\left\{
			\begin{array}{l}
				\dot{y}(t) \geq Cy(t)^p, \quad 0<t_0<t<t_1, \label{blowupODE} \\ 
				y(t_0) = M_0,
			\end{array}\right.
		\end{equation}
	where $C,\,M_0>0$. We can integrate \eqref{blowupODE} explicitly to obtain
		\begin{equation*}
			0\leq y(t)^{1-p} \leq C(1-p)(t-t_0) + M_0^{1-p}.
		\end{equation*}
	Hence, $y(t)^{1-p} \rightarrow 0$ as $t \rightarrow t_0 + \frac{M_0^{1-p}}{C(p-1)}$. Since $1-p<0$, this implies $y(t) \rightarrow +\infty$. Alternatively, for a fixed $t_1>t_0$, blow-up occurs for $t<t_1$ provided we have 
		\begin{equation}\label{choiceM0}
			M_0 > \left[C(p-1)(t_1-t_0)\right]^{-\frac{1}{p-1}}.
		\end{equation}
	
	So fix $T>0$ and assume that the viscosity solution $u$ of \eqref{mainEq} satisfies \eqref{boundarydata} in the classical sense. We will later specify the largeness condition on $u_0$ in terms of $M_0$ above, but may consider it set from now on, since it depends only on constants already available. The constant $C$ in \eqref{blowupODE} is also specified later, depending only on the appropriate quantities. In particular, it is independent of both $\eta$ and $u_0$.

	Recall the approximate equation obtained in Proposition \ref{secAppIneq},
		\begin{equation*}
			w_t + (-\Delta)^s w - |Dw|^p \geq 0 \quad \textrm{a.e.~in } \Omega^\eta \times (t_0, t_1),
		\end{equation*}
	where now $w$ is obtained by regularization of the viscosity solution $u$ of \eqref{mainEq}. Define $z(t) = \int_{\Omega^\eta} w(x,t) \varphi_1^\eta(x) \,dx,$ where $\varphi_1^\eta$ is the unique positive solution of \eqref{fracLapEigenval} normalized so that $\|\varphi_1^\eta\|_\infty=1$. From Proposition \ref{supInfConvProperties} \eqref{unifConverApp} and Remark \ref{uBounded}, we obtain, for sufficiently small $\eta$, that $\|w\|_\infty\leq \|u\|_\infty + 1 \leq \|u_0\|_\infty + 1$. Thus $z(t)$ is uniformly bounded for $0\leq t \leq T$. In what remains of the proof, we show that $z$ satisfies \eqref{blowupODE} by using the assumption that the solution $u$ satisfies \eqref{boundarydata} in the classical sense, and hence blows-up, a contradiction.
	
	From \eqref{AppIneq}, we compute
		\begin{align*}
			\dot{z}(t) &{}= \int_{\Omega^\eta} w_t(x,t) \varphi_1^\eta(x) dx \geq \int_{\Omega^\eta} \left[-(-\Delta)^s w(x,t) + |Dw(x,t)|^p\right]\varphi_1^\eta(x) \,dx\\
			&{}= -\int_{\Omega^\eta} (-\Delta)^s w(x,t) \varphi_1^\eta(x) \,dx + \int_{\Omega^\eta} |Dw(x,t)|^p \varphi_1^\eta(x) \,dx =: I_1 + I_2.
		\end{align*} 
	We proceed to ``integrate by parts''in $I_1$, taking some care to handle the $P.V.$ in the definition of the fractional Laplacian.
	
	Since $\varphi_1^\eta \equiv 0$ in $\mathbb{R}^N\backslash 
	\Omega^\eta$ (point-wise), we have
		\begin{align*}
			I_1 &{}= -\int_{\Omega^\eta} (-\Delta)^s w(x,t) \varphi_1^\eta(x) \,dx = -\int_{\mathbb{R}^N} (-\Delta)^s w(x,t) \varphi_1^\eta(x) \,dx \\
			&{}= -C_{N,s}\int_{\mathbb{R}^N} \lim_{\omega\to 0} \int_{\mathbb{R}^N\backslash B_\omega(x)} \frac{w(x,t) - w(y,t)}{|x-y|^{N+2s}} \varphi_1^\eta(x) \,dy\,dx.
		\end{align*}
	Since $w(\cdot, t)\in C^{1,1}(\mathbb{R}^N)$ for all $t\in [t_0,t_1]$, $(-\Delta)^sw(\cdot, t)$ is classically defined a.e.~in $\Omega$, precisely at the points where $w$ has a second order expansion. Moreover, at such points we have that the integral with respect to $y$ converges as $\omega\to 0$ and, by the standard computation,
		\begin{align*}
			\left|\int_{\mathbb{R}^N\backslash B_\omega(x)} \frac{w(x,t) - w(y,t)}{|x-y|^{N+2s}} \varphi_1^\eta(x) \,dy\, \right| \leq C_{N,s}\|D^2w\|_\infty\varphi_1^\eta(x).
		\end{align*}
	Hence by the dominated convergence theorem,
		\begin{equation*}
			I_1 = -C_{N,s}\lim_{\omega\to 0} \int_{\mathbb{R}^N} \int_{\mathbb{R}^N\backslash B_\omega(x)} \frac{w(x,t) - w(y,t)}{|x-y|^{N+2s}} \varphi_1^\eta(x) \,dy\,dx.
		\end{equation*}
	We note that the integrand is no longer singular, and write $dV$ for the measure on $\mathbb{R}^{2N}$. Applying first Fubini's theorem, and then by symmetry (i.e., interchanging $x$ and $y$), we have
		\begin{align*}
			I_1 &{}= -C_{N,s}\lim_{\omega\to 0} \int_{\{|x-y|\geq \omega\}} \frac{w(x,t) - w(y,t)}{|x-y|^{N+2s}} \varphi_1^\eta(x) \,dV\\
			&{}= -C_{N,s}\lim_{\omega\to 0} \int_{\{|x-y|\geq \omega\}} \frac{w(x,t) - w(y,t)}{|x-y|^{N+2s}} (-\varphi_1^\eta(y)) \,dV\\
			&{}= -\frac{C_{N,s}}{2}\lim_{\omega\to 0} \int_{\{|x-y|\geq \omega\}} \frac{(w(x,t) - w(y,t))(\varphi_1^\eta(x) - \varphi_1^\eta(y))}{|x-y|^{N+2s}}\,dV.
		\end{align*}
	Starting from the last integral, we repeat the above computation ``in reverse'' to pass the operator onto $\varphi_1^\eta$ and use the associated eigenvalue problem \eqref{fracLapEigenval}. Since now $w\equiv 0$ in $\mathbb{R}^N\backslash\Omega$, this gives
		\begin{align*}
			I_1 &{}= -\frac{C_{N,s}}{2}\lim_{\omega\to 0} \int_{\{|x-y|\geq \omega\}} \frac{(w(x,t) - w(y,t))(\varphi_1^\eta(x) - \varphi_1^\eta(y))}{|x-y|^{N+2s}}\,dV\\
			&{}= -\int_{\mathbb{R}^N} w(x,t) (-\Delta)^s \varphi_1^\eta(x) \,dx = -\int_{\Omega} w(x,t) (-\Delta)^s \varphi_1^\eta(x) \,dx\\
			&{}= -\int_{\Omega^\eta} w(x,t) (-\Delta)^s \varphi_1^\eta(x) \,dx -\int_{\Omega\backslash\Omega^\eta} w(x,t) (-\Delta)^s \varphi_1^\eta(x) \,dx\\
			&{}= -\lambda_1^\eta z(t) - \int_{\Omega\backslash\Omega^\eta} w(x,t) (-\Delta)^s \varphi_1^\eta(x) \,dx.
		\end{align*}
	Using that $w,\,\varphi_1^\eta\geq 0$, and again that $\varphi_1^\eta \equiv 0$ in $\mathbb{R}^N\backslash 
	\Omega^\eta$,
		\begin{align*}
			& \int_{\Omega\backslash\Omega^\eta} w(x,t) (-\Delta)^s \varphi_1^\eta(x) \,dx = - \int_{\Omega\backslash\Omega^\eta} w(x,t) \int_{\mathbb{R}^N}\frac{\varphi_1^\eta(x) - \varphi_1^\eta(y)}{|x-y|^s} \,dy\,dx\\
			&\quad = \int_{\Omega\backslash\Omega^\eta} w(x,t) \int_{\Omega^\eta}\frac{ - \varphi_1^\eta(y)}{|x-y|^s} \,dy\,dx
			= -\int_{\Omega\backslash\Omega^\eta}\int_{\Omega^\eta} \frac{w(x,t)\,\varphi_1^\eta(y)}{|x-y|^s} \,dy\,dx \leq 0.
		\end{align*}
	Therefore, 
		\begin{equation}\label{I1}
		I_1\geq -\lambda_1^\eta z(t).
		\end{equation}
	
	We turn to estimating $I_2$. First, applying H\"older's inequality and Lemma \ref{finiteIntLemma},
		\begin{align}\label{holder}
			&\int_{\Omega^\eta} |Dw| \,dx = \int_{\Omega^\eta} |Dw| (\varphi_1^\eta)^\frac{1}{p}(\varphi_1^\eta)^{-\frac{1}{p}} \,dx\nonumber\\
			&\quad \leq \left(\int_{\Omega^\eta} |Dw|^p \,\varphi_1^\eta \,dx\right)^\frac{1}{p}\,\left(\int_{\Omega^\eta} (\varphi_1^\eta)^{-\frac{1}{p-1}} \,dx\right)^\frac{p-1}{p} \leq C_4^{\frac{p-1}{p}}\,I_2^\frac{1}{p},
		\end{align}
	where $C_4$ is the constant from Lemma \ref{finiteIntLemma}.
	
	We apply the following version of Poincaré's inequality to $w(\cdot,t)\in C^{1,1}(\Omega^\eta)$: 
		\begin{equation}\label{poincareIneq}
			\int_{\Omega^\eta} |w(\cdot,t)| \,dx \leq C(\Omega)\left(\sup_{\partial\Omega^\eta} |w(\cdot,t)| + \int_{\Omega^\eta} |Dw(\cdot,t)| \,dx\right),
		\end{equation}
	where the constant appearing on the right-hand side can be given in terms of the diameter of the domain, and can therefore be taken uniformly with respect to $\eta$.
	
	From the uniform convergence of the approximation $w\to u$ (Proposition \ref{supInfConvProperties} \eqref{unifConverApp}) and the uniform continuity of $u$ in $\overline{\Omega}\times[0,T]$, we have that $\sup_{\partial\Omega^\eta} |w(\cdot,t)|\to 0$ as $\eta\to 0$, uniformly in $t$. Indeed, let $\nu > 0$. The assumption that $u$ satisfies \eqref{boundarydata} point-wise implies in particular that $u(\cdot,t)=0$ in $\partial\Omega$ for all $t\geq t_0$. Therefore, for any $x\in\partial\Omega^\eta$, if $x_0\in\partial\Omega$ is such that $d(x,x_0)=d(x,\partial\Omega)$, we have
		\begin{align}\label{supWonDeta}
			w(x,t) ={}& w(x,t) - u(x_0,t) \leq |w(x,t) - u(x,t)| \nonumber\\
			&{}\ + |u(x,t) - u(x_0,t)| < 2\nu.
		\end{align}
	Hence, we later write simply $\sup_{\partial\Omega^\eta} |w(\cdot,t)|=o(1)$ as $\eta\to 0$. Together with the normalization $\|\varphi_1^\eta\|_\infty = 1$ and the elementary inequality $(a+b)\leq 2^{p-1}(a^p + b^p)$, \eqref{holder}, \eqref{poincareIneq} and \eqref{supWonDeta} imply
		\begin{align*}
			|z(t)|^p &{}\leq \left|\int_{\Omega^\eta} w(x,t)\varphi_1^\eta(x)\,dx\right|^p \leq \left(\|\varphi_1^\eta\|_\infty \int_{\Omega^\eta} |w(x,t)| \,dx \right)^p\\
			&{}\leq C(\Omega)\left(\sup_{\partial\Omega^\eta} |w| + \int_{\Omega^\eta} |Dw| \,dx\right)^p \leq o(1) + C\left(\int_{\Omega^\eta} |Dw| \,dx\right)^p\\
			&{}\leq o(1) + CI_2.
		\end{align*}
	Recalling \eqref{holder} and \eqref{I1}, we obtain
		\begin{equation}\label{zDot}
			\dot{z}(t) \geq -\lambda_1^\eta z(t) + C_5 z(t)^p + o(1),
		\end{equation}
	where the $C_5>0$ does not depend on either $\eta$ or $u_0$; this is crucial for the next step.

	We can reduce \eqref{zDot} to \eqref{blowupODE} as follows. We use that $\varphi_1^\eta \to \varphi_1$ uniformly over $\overline{\Omega}$ as $\eta\to 0$; $w \to u$ uniformly over $\overline{\Omega}\times[0,T]$ as $\eta\to 0$ and $t_0,\,T-t_1 \to 0$ (see Proposition \ref{supInfConvProperties}\eqref{unifConverApp} and Proposition \ref{secAppIneq}); that $u(\cdot, t_0)\to u_0$ as $t_0\to 0$, and the fact that all these functions are uniformly bounded in $\eta$, to conclude
		\begin{align*}
			z(t_0) ={}& \int_{\Omega^\eta} w(x,t_0)\, \varphi_1^\eta(x) \,dx = \int_{\Omega} w(x,t_0)\, \varphi_1^\eta(x) \,dx + \int_{\Omega\backslash\Omega^\eta} w(x,t_0)\, \varphi_1^\eta(x) \,dx\\
			={}&\int_{\Omega} w(x,t_0)\, \varphi_1^\eta(x) \,dx + o(1) = \int_{\Omega} u_0(x)\,\varphi_1(x) \,dx + o(1) \quad\textrm{as }\eta\to 0.
		\end{align*}
	Thus we see that the largeness condition for $u_0$ given in Theorem \ref{nonExtThm} implies that $z(t_0)$ is large as well. On the other hand, assumption \eqref{pands} implies in particular that $p>1$, hence $z(t)^p$ is the dominating term in the right-hand side of \eqref{zDot}. Therefore, taking
		\begin{equation*}
			z(t_0) \geq \max\left\{M_0, \left(\frac{2\lambda_1}{C_5}\right)^\frac{1}{p-1}\right\} + 1,
		\end{equation*}
	and $\eta$ sufficiently small ensures that both $\dot{z}(t)\geq \frac{C_5}{2}z(t)^p$ for $t>t_0$ and that $z(t_0)\geq M_0$, which together are equivalent to \eqref{blowupODE}. This gives the desired contradiction. \end{proof}

	\begin{remark}\label{nonExtRmk1}
		Given the indirect nature of the preceding proof, we would like to highlight the role played by the main assumptions leading to LOBC: the fact that the gradient term is ``dominating'' in the equation, i.e., $p>s+1$, from \eqref{pands}, is used only in Lemma \ref{finiteIntLemma}. On the other hand, the assumption that leads to contradiction, that \eqref{boundarydata} is satisfied in the classical sense, is used only in \eqref{supWonDeta} and in applying Poincaré's inequality.
		
		The case of more general boundary conditions can be treated in exactly the same way as above. Assuming $u=g$ in $\mathbb{R}^N\backslash\Omega\times (0,T)$ is satisfied in the classical sense, with $g\in C_b(\mathbb{R}^N\backslash\Omega\times(0,T))$, we obtain
			\begin{equation*}
				\dot{z}(t) \geq -\lambda_1^\eta z(t) + Cz(t)^p - C_6,
			\end{equation*}
		instead of \eqref{zDot}, where $C_6$ depends on $\|g\|_{L^\infty(\partial\Omega\times(0,T))}$. From here on, the proof continues as above.
	\end{remark}
	
	\begin{remark}
		Assuming higher regularity for the initial data, e.g., $u_0\in C^2(\overline{\Omega})$ it is possible to obtain estimates for $u_t$ (see, e.g., \cite{tabet2010large}, Proposition 4.1, for an example of this method in the local setting). This allows the application of regularity results available for stationary problems (e.g, those of \cite{barles2016existence}) to our problem, essentially by treating $u_t$ as a bounded ``right-hand side''. Global H\"older estimates can then be obtained for the solution of \eqref{mainEq}-\eqref{boundarydata}-\eqref{initialdata}. In this case, Theorems \eqref{localExtThm} and \eqref{nonExtThm} together imply that for any $\beta>\beta^*=\frac{p-2s}{p-1}$, the H\"older semi-norm of the solution $u$ \emph{blows-up in finite time}, i.e., there exists $T_1>0$, depending only on $N, \Omega, s, p$ and $u_0$, such that
			\begin{equation*}
				\lim_{t\to T_1} [u(\cdot,t)]_{\beta,\overline{\Omega}} = \infty.
			\end{equation*}
		This situation is analogous to that of \emph{gradient blow-up} for \eqref{vhj}.
	\end{remark}

%% file: appendix_Attainment_v4.tex
\appendix
\section{Uniform $C^s$ regularity for the approximate domains}\label{CsReg}
	
	In this Appendix we state a version of results from \cite{ros2014dirichlet} which concern the regularity of solutions to the Dirichlet problem for the fractional Laplacian. We revisit the corresponding proofs to show that the estimates are uniform with respect to varying domains such as those appearing in \eqref{fracLapEigenval}. In this way we conclude the analysis postponed in the proof of Theorem \ref{globalReg}.

	Let $s\in(0,1)$, $g\in L^\infty(\Omega^\eta)$ and consider
		\begin{equation}\label{fracDirichlet}
			\left\{\begin{array}{ll}
				(-\Delta)^s u = g & \textrm{in } \Omega^\eta\\
				u = 0 & \textrm{in } \mathbb{R}^N \backslash\Omega^\eta,
			\end{array}\right.
		\end{equation}
	where $\Omega$ is a bounded, $C^2$ domain, $\Omega^\eta = \{x\in\Omega\,| \, d(x) >\eta \}$, $\eta\in(0,\eta_0)$ and $\eta_0$ is given by \ref{choiceEta0}.
	
	\begin{proposition}\label{globalDirReg}
		Let $u$ be a solution of \eqref{fracDirichlet}. Then $u\in C^s(\mathbb{R}^N)$ and
			\begin{equation*}
				\|u\|_{C^s(\mathbb{R}^N)} \leq C\|g\|_{L^\infty(\Omega^\eta)},
			\end{equation*}
		where $C$ is a constant depending only on $\Omega$ and $s$. In particular, the constant $C$ can be taken uniformly for $\eta\in(0,\eta_0)$.
	\end{proposition}

This is Proposition 1.1 from \cite{ros2014dirichlet}, save for the dependency of $C$ on the parameter $\eta$, which we require to be uniform. To this end, we outline the manner in which this result was obtained. We begin stating the key Lemmas leading up to it.
	
	\begin{lemma}\label{1stIntHolderLemma}
		Assume that $w\in C^\infty(\mathbb{R}^N)$ is a solution of $(-\Delta)^s w = h$ in $B_2$. Then, for every $\beta\in(0,2s)$,
			\begin{equation*}
				\|w\|_{C^\beta(\overline{B_\frac{1}{2}})} \leq C \left(\|(1+|x|^{-N-2s}w(x)\|_{L^1(\mathbb{R}^N)} + \|w\|_{L^\infty(B_2} + \|h\|_{\infty(B_2)} \right),
			\end{equation*}
		where the constant $C$ depends only on $N, s$ and $\beta$.
	\end{lemma}
	
	\begin{proof}
		This is Corollary 2.5 in \cite{ros2014dirichlet}.
	\end{proof}
	
	\begin{lemma}\label{SupersolnLemma}
		There exists a positive constant $C$ and a radial continuous function $\phi_1\in H^s_{\textit{loc}}(\mathbb{R}^N)$ satisfying
			\begin{equation}\left\{
				\begin{array}{ll}
					(-\Delta)^s\phi_1 \geq 1 & \textrm{in } B_4\backslash B_1\\
					\phi_1\equiv 0 & \textrm{in } B_1\\
					0\leq \phi_1 \leq C (|x|-1)^s & \textrm{in } B_4\backslash B_1\\
					1 \leq \phi_1 \leq C &\textrm{in } \mathbb{R}^N\backslash B_4.
				\end{array}\right.
			\end{equation}
	\end{lemma}
	
	\begin{proof}
		This is Lemma 2.6 in \cite{ros2014dirichlet}.
	\end{proof}
	
	\begin{remark}
		Lemmas \ref{1stIntHolderLemma} and \ref{SupersolnLemma} have no dependence on the domains $\Omega$ nor $\Omega^\eta$. Therefore, they apply directly to our setting.
	\end{remark}
	
	\begin{lemma}\label{roslemma2.7}
		Let $\Omega^\eta$ and $g$ be as above, and let $u$ be the solution of \ref{fracDirichlet}. Then
			\begin{equation*}
				|u(x)| \leq C\|g\|_{L^\infty(\Omega^\eta)}d_\eta(x)^s \quad \textrm{for all } x\in\Omega^\eta,
			\end{equation*}
		where $C$ depends only on $\Omega$ and $s$. In particular, $C$ can be taken uniformly in $\eta\in(0,\eta_0)$.
	\end{lemma}
		
		Lemma \ref{roslemma2.7} relies on the following result:
	
	\begin{lemma}\label{rosclaim2.8}
		Let $\Omega$ be a bounded domain and let $g\in L^\infty(\Omega^\eta)$. Let $u$ be the solution of \eqref{fracDirichlet}. Then
			\begin{equation*}
				\|u\|_{L^\infty(\mathbb{R}^N)} \leq C(\mathrm{diam} (\Omega^\eta) )^{2s}\|g\|_{L^\infty(\Omega^\eta)},
			\end{equation*}
		where $C$ is a constant depending only on $N$ and $s$.
	\end{lemma}
	
	\begin{proof}
		This is Claim 2.8 in \cite{ros2014dirichlet}. Since $\mathrm{diam}(\Omega^\eta)\leq\mathrm{diam}(\Omega)$, the estimate is uniform for $\eta\in(0,\eta_0)$.
	\end{proof}
	
	\begin{proof}[Proof of Lemma \ref{roslemma2.7}:] 
		This is Lemma 2.7 in \cite{ros2014dirichlet}. For points near $\partial\Omega^\eta$, the estimate is obtained by scaling the supersolution from Lemma \ref{SupersolnLemma} to the annular region $B_{2\rho_0}\backslash B_{\rho_0}$, where $B_{\rho_0}$ is an exterior tangent ball to $\partial\Omega$, and applying comparison. Owing to Remark \ref{uniUniExtBall}, the scaling can be done uniformly with respect to $\eta\in(0,\eta_0)$. For the remaining points in $\Omega^\eta$, Lemma \ref{rosclaim2.8} is employed.
	\end{proof}
	
	\begin{lemma}\label{roslemma2.9}
		Let $\Omega$ be a bounded domain satisfying the exterior ball condition, $g\in L^\infty(\Omega^\eta)$, and $u$ be the solution of \ref{fracDirichlet}. Then $u\in C^s(\Omega^\eta)$ and for all $x_0\in\Omega^\eta$, 
			\begin{equation}
				[u]_{C^s(\overline{B_{R}(x_0)})} \leq C\| g\|_{L^\infty(\Omega^\eta)},
			\end{equation}
		where $R=\frac{d(x_0)}{2}$ and $C$ depends only on $\Omega^\eta$, and $s$.
	\end{lemma}
	
	\begin{proof}
		This is a special case of Lemma 2.9 in \cite{ros2014dirichlet}. Although more intricate than that of the previous results, the proof of this result uses only a scaling of the interior estimate of Lemma \ref{1stIntHolderLemma} to the ball $B_R(x_0)$, the use of the upper barrier for $\|u\|_{L^\infty(\Omega)}$ obtained in \ref{roslemma2.7}, and a covering argument. As such, the constant $C$ in Lemma \ref{roslemma2.9} now depends on the measure of $\Omega$ as well. This quantity, however, varies continuously for $\Omega^\eta$ with $\eta\in(0,\eta_0)$.
	\end{proof}
	
	\begin{proof}[Proof of Proposition \ref{globalDirReg}:]
		It remains only to extend the estimate from Lemma \ref{roslemma2.9} up to the boundary. For this we provide an argument from \cite{iannizzotto2016global}. Through a covering argument, Lemma \ref{roslemma2.9} extends to an interior bound on any compact subset of $\Omega^\eta$. Consider now $x, y \in \Omega^\eta$ such that $0\leq d_\eta(x),\, d_\eta(y) \leq \rho$ for a small $\rho>0$ and, without loss of generality, $d_\eta(y) \leq d_\eta(x)$. There are two possible cases:
			\begin{enumerate}[a)]
				\item $2|x-y|\leq d_\eta(x)$. This implies that $y\in B_{R}(x)$ with $R=\frac{d_\eta(x)}{2}$. Hence, applying Lemma \ref{roslemma2.9}, we obtain $|u(x) - u(y)| \leq C\|g\|_{L^\infty(\Omega^\eta)}|x-y|^s$.
				
				\item $2|x-y| > d_\eta(x) \geq d_\eta(y)$. In this case, we apply Lemma \ref{roslemma2.7} to compute
					\begin{equation*}
						u(x) - u(y) \leq  |u(x)| + |u(y)| \leq C (d_\eta(x)^s + d_\eta(y)^s) \leq C |x-y|^s.
					\end{equation*}
			\end{enumerate}
	\end{proof}